\documentclass[11pt,twoside]{article}
\usepackage[utf8]{inputenc}
\usepackage{amsthm,amsmath,amssymb,epsf}
\usepackage{graphicx}
\usepackage{xcolor}
\usepackage{enumitem}

\newcommand{\xx}{{\bf x}}
\newcommand{\yy}{{\bf y}}

\newcommand{\NN}{{\mathbb N}}

\newcommand{\RR}{{\mathbb R}}

\newcommand{\cC}{{\cal C}}

\newcommand{\cL}{{\cal L}}

\newcommand{\cT}{{\cal T}}

\newcommand{\1}{{\bf 1}}

\newcommand{\ds}{\displaystyle}

\newtheorem{defn}{Definition}
\newtheorem{thm}{Theorem}
\newtheorem{lemma}{Lemma}
\newtheorem{cor}{Corollary}

\newtheorem{rem}{Remark}
\newtheorem{algo}{Algorithm}
\renewenvironment{proof}{\noindent{\bf Proof:} }{\hfill $\square$ \\ }

\begin{document}
\thispagestyle{empty}

\begin{center}
{\LARGE Shadow Simulated Annealing algorithm: a new tool for global optimisation and statistical inference}\\[.5in]

{\large  R. S. Stoica$^{1}$}, M. Deaconu$^{2}$, A. Philippe$^{3}$, L. Hurtado$^{4}$\\[0.6in]

{\large $^{1}$ Université de Lorraine, CNRS, IECL, F-54000, Nancy, France}\\
{\large $^{2}$ Université de Lorraine, CNRS, Inria, IECL, F-54000, Nancy, France}\\
{\large $^3$ Université de Nantes, Laboratoire de Mathématiques Jean Leray, Nantes, France}\\
{\large $^{4}$ Departamento de Matem\'atica Aplicada y Estad\'{\i}stica, Universidad CEU San Pablo, 28003 Madrid, Spain}
\end{center}

\begin{verse}
{\footnotesize \noindent {{\bf Abstract}: This paper develops a new global optimisation method that applies to a family of criteria that are not entirely known. This family includes the criteria obtained from the class of posteriors that have normalising constants that are analytically not tractable. The procedure applies to posterior probability densities that are continuously differentiable with respect to their parameters. The proposed approach avoids the re-sampling needed for the classical Monte Carlo maximum likelihood inference, while providing the missing convergence properties of the ABC based methods. Results on simulated data and real data are presented. The real data application fits an inhomogeneous area interaction point process to cosmological data. The obtained results validate two important aspects of the galaxies distribution in our Universe~: proximity of the galaxies from the cosmic filament network together with territorial clustering at given range of interactions. Finally, conclusions and perspectives are depicted.\\
}}
\end{verse}

\section{Introduction}
\noindent
A large class of the mathematical questions issued from data sciences can be formulated as an optimisation problem. The complexity of the data structures requires more and more elaborate models with an important number of parameters controlling different aspects outlined by the data observation process. Within this context, a typical a question is, what is the most probable model able to reproduce the behaviour exhibited by the analysed data. Clearly, the possible answer to this question assumes the existence of a class of models together with priors associated to its parameters.\\

\noindent
A natural way to answer this question is the computation of the global maximum of the induced posterior distribution. The classical simulated annealing framework proposes the solution to this problem, provided sampling from the posterior distribution is possible.\\

\noindent
Sampling posterior distributions is still a challenging mathematical problem. This is due to the fact that sometimes, the proposed simulation algorithms may be required to make computations of analytical intractable quantities. An example of such quantities is the evaluation of normalisation constants of the probability densities describing the considered models.\\

\noindent
If only parameters estimation is considered, the common solution adopted is to provide maximum likelihood computations based on Monte Carlo simulations. This framework allows the user to benefit of the whole theoretical power of the likelihood inference, if one afford the price to pay in terms of computational costs. The computational cost may be excessively high whenever the initial condition is to far away from the desired solution. Furthermore this phenomenon introduces numerical instability of the proposed solution. The only reliable strategy within this context is to re-sample the model, as often as possible, hence increasing the computational cost.\\

\noindent
Since less than a decade, a new methodological framework for statistical inference, the Approximate Bayesian Computation (ABC) allows to extend the Monte Carlo likelihood based inference, by providing solutions for sampling approximately from the posterior distribution. The authors in~\cite{StoiEtAl17} proposed a new algorithm, ABC Shadow, that overcomes the main drawback of the previously mentioned ABC methods: the ABC Shadow allows the output distribution of the algorithm, to be as closed as desired to the aimed posterior distribution.\\

\noindent
The work presented in this paper leads directly to a new method of parameter estimation based on a simulated annealing algorithm. To better outline its interest, let us consider the following example, inspired by applications in spatial data analysis.\\

\noindent
Let $\yy$ be an object pattern that is observed in a compact window $W \subset \RR^{d}$. The observed pattern is supposed to be the realisation of  a spatial process. Such a process is given by the probability density
\begin{equation}
p(\yy|\theta)=\frac{\exp[-U(\yy|\theta)]}{\zeta(\theta)}
\label{gibbsDistribution}
\end{equation}
with $U(\yy|\theta)$ the energy function and $\zeta(\theta)$ the normalising constant. The model given by~\eqref{gibbsDistribution} may be considered as a Gibbs process, and it may represent a random graph, a Markov random field or a marked point process. Let $p(\theta|\yy)$ be the conditional distribution of the model parameters or the posterior law
\begin{equation} 
p(\theta|\yy) = \frac{\exp[-U(\yy|\theta)]p(\theta)}{Z(\yy)\zeta(\theta)},
\label{posteriorGibbs}
\end{equation}
where $p(\theta)$ is the prior density for the model parameters and $Z(\yy)$ the normalising constant. The posterior law is defined on the parameter space $\Theta$. For simplicity, the parameter space is considered to be a compact region in $\RR^{r}$ with $r$ the size of the parameter vector. Let $\nu$ be the corresponding Lebesgue measure. The parameter space is endowed with its Borel algebra $\cT$.\\

\noindent
In the following, it is assumed that the probability density $p(\yy|\theta)$ is strictly positive and continuously differentiable with respect to $\theta$. This hypothesis is strong but keeps realistic, since it is often required by practical applications.\\

\noindent
This paper constructs and develops a simulated annealing method to compute~:
\begin{equation*}
\widehat{\theta} = \arg \max_{\theta \in \Theta} p(\theta|\yy).
\end{equation*}
The difficulty of the problem is due to the fact that the normalising constant $\zeta(\theta)$ is not available in analytic closed form. Hence, special strategies are required to sample from the posterior distribution~\eqref{posteriorGibbs} in order to implement an optimisation procedure.\\

\noindent
The plan of the paper is as follows. First, an Ideal Markov chain (IC) is constructed. This chain has as equilibrium distribution the posterior distribution of interest. So, in theory, this chain can be used to sample from the distribution of interest. Next, an Ideal Simulated Annealing (ISA) process built using a non-homogeneous IC chain is constructed. The convergence properties of the process are also given. Despite the good theoretical properties, the ISA process cannot be used to build algorithms for practical use, since it requires the computation of normalising constants of the form $\zeta(\theta)$. The fourth section presents a solution to this problem. First an approximate sampling mechanism called the Shadow chain (SC) is presented. The SC  is able to follow closely within some fixed limits the IC. Based on the SC chain, a Shadow Simulated Annealing (SSA) process is built. This process is controlled by two parameters, evolving slowly to zero. This double control allows to derive convergence properties of the process towards the global optimum we are interested in. These theoretical results allow the construction of a SSA algorithm. The algorithm is applied to simulated and real data, during the fifth section. The real data application fits an inhomogeneous point process with interactions to a cosmological data set, in order to obtain essential characteristics of the galaxies distribution in our Universe. At the end of the paper, conclusions and perspectives are formulated.\\

\section{Ideal Chain for posterior sampling}
\noindent
In theory, Markov Chain Monte Carlo algorithms may be used for sampling $p(\theta | \yy)$. For instance, let us consider the general Metropolis  Hasting algorithm. Assuming the system is in the state $\theta$, this algorithm first chooses a new value $\psi$ according to a proposal density $q(\theta \rightarrow \psi)$. The value $\psi$ is then accepted with probability $\alpha_{i}(\theta \rightarrow \psi)$ given by
\begin{equation}
\alpha_{i}(\theta \rightarrow
\psi)=\min_{\psi}\left\{1,\frac{p(\psi|\yy)}{p(\theta|\yy)}\frac{q(\psi
\rightarrow \theta)}{q(\theta \rightarrow \psi)}\right\}.
\label{acceptance_probability_ideal}
\end{equation}

\noindent 
The transition kernel of the Markov chain simulated by this algorithm is given by, for every $ A \in  \cT$
\begin{equation}
\label{kernel}
\begin{array}{ll}
P_i(\theta,A) & = \ds\int_{A}\alpha_{i}(\theta \rightarrow
\psi)q(\theta \rightarrow \psi)\1_{\{\psi \in A\}}d\psi\\
& \quad + \, \1_{\{\theta \in A\}}
\left[1 - \ds\int_A   \alpha_i(\theta\rightarrow \psi) q(\theta \rightarrow \psi) d\psi \right].
\end{array}
\end{equation}

\noindent
 Let us recall that the transition kernel of a Markov chain may act on both functions and measures, as it follows
\begin{equation}
    \label{kernelP} 
    P_{i}f(x) =\ds\int_\Theta  P_{i}(x,dy)f(y), \quad  \mu P_{i}(A) = \int_{\Theta}\mu(d\theta)P_{i}(\theta,A).
\end{equation}
\noindent
The conditions that the proposal density $q(\theta \rightarrow \psi)$ has to meet, so that the simulated Markov chain has a unique equilibrium distribution
\begin{equation*}
\pi(A)=\int_{A}p(\theta|\yy)d\nu(\theta),
\end{equation*}
\noindent
are rather mild~\cite{Tier94}. Furthermore, if $q$ and $\pi$ are bounded and bounded away from zero on the compact $\Theta$, then the simulated chain is uniformly ergodic~(\cite{Tier94} Prop. 2, \cite{RobeTwee96} Thm. 2.2). Hence, there exist a positive constant $M$ and a positive constant $\rho < 1$ such that
\begin{equation*}
\sup_{\theta \in \Theta}\parallel P_{i}^{n}(\theta,\cdot) - \pi(\cdot) \parallel \leq M\rho^{n}, \quad n \in \NN.
\end{equation*}

\noindent
For a fixed $\delta > 0$, a parameter value $\nu \in \Theta$ and a realisation $\xx$ of the model $p(\cdot|\nu)$ given by~\eqref{gibbsDistribution}, let us consider the proposal density
\begin{eqnarray}
q(\theta \rightarrow \psi) = q_\delta (\theta \rightarrow \psi |
\xx) = \frac{f(\xx |\psi)/\zeta(\psi)}{I(\theta, \delta,
\xx)}\1_{b(\theta, \delta/2)}\{\psi\} 
\label{idealProposal}
\end{eqnarray}
with $f(\xx|\psi) = \exp[-U(\xx|\psi)]$. Here $\1_{b(\theta, \delta/2)}\{\cdot\}$ is the indicator function over $b(\theta, \delta/2)$, which is the ball of center $\theta$ and radius $\delta/2$. Finally, $I(\theta, \delta, \xx)$ is the quantity given by the integral 
\begin{equation*}
I(\theta, \delta, \xx) = \int_{b(\theta, \delta/2)} \ds\frac{f(\xx|\phi)}{\zeta(\phi)} \, d \phi.
\end{equation*}

\noindent
This choice for $q(\theta \rightarrow \psi)$ guarantees the convergence of the chain towards $\pi$ and avoids the evaluation of the normalising constant ratio $\zeta(\theta)/\zeta(\psi)$ in~\eqref{acceptance_probability_ideal}. We call the chain induced by these proposals the {\it ideal chain}. Nevertheless, the proposal~\eqref{idealProposal} requires the computation of integrals such as $I(\theta, \delta, \xx)$, and this is as difficult as the computation of the normalising constant ratio. Later in the paper it will be shown how this construction allows a natural approximation of the ideal chain: the shadow chain.

\section{Ideal Simulated Annealing process}
\subsection{Principle}
\noindent
The construction and the properties of a {\it Simulated Annealing} (SA) process in a general state space were investigated by~\cite{HaarSaks91,HaarSaks92}. \\

\noindent
The SA process is built with the following ingredients: a function to optimise $h$, a Markov transition kernel $P$ with equilibrium distribution $\pi \propto \exp(-h)$ and a cooling schedule for the temperature parameter $T$.\\

\noindent
The function $h$ is assumed to be continuous differentiable in $\theta$ and its global maximum is $\theta_{opt}$. In the present situation $h : \Theta \rightarrow \RR$ is obtained by taking the logarithm of~\eqref{posteriorGibbs}:
\begin{equation}
h(\theta) =  U(\yy|\theta) + \log \zeta(\theta) - \log p(\theta) + \log Z(\yy).
\label{functionH}
\end{equation}
It represents the loglikelihood function to which the prior term $\log p(\theta)$ is added. Clearly, if $p(\theta)$ is the uniform distribution over $\Theta$, then maximising $h$ leads to the maximum likelihood estimation.  The transition kernel $P$ is given by \eqref{kernel}. The cooling schedule for the temperature is a logarithmic one, and it results from the proofs of the SA convergence.\\

\noindent
The SA process simulates iteratively a sequence of distributions $\pi_n$
\begin{equation}
\label{pin}
\pi_n(A)=
\ds\int_A p_{n}(\theta)d\nu(\theta) =
\ds\int_A\ds\frac{\exp(-h(\theta)/T_n)}{c_n}d\nu(\theta) 
\end{equation}
with $p_n(\theta) = \frac{\exp(-h(\theta)/T_n)}{c_n}$ and $c_n=\int_\Theta
\exp(-h(\theta)/T_n)d\nu(\theta)$, while $T_n$ goes slowly to zero.\\

\noindent
Each distribution $\pi_n$ is simulated using a transition kernel $P_n$ having it as equilibrium distribution. The transition kernel $P_n$ is obtained, by modifying~\eqref{kernel} in order to sample from $\pi_n$. The kernels sequence $(P_n)_{n \geq 0}$ induces an inhomogeneous Markov chain.\\

\noindent
At low temperatures, the process converges weakly towards the global optimum of $h$, that is 
\begin{equation}
\label{pin-1}
\pi_n \underset{n \rightarrow \infty}{\longrightarrow} \delta_{\theta_{opt}}
\end{equation}
with $\delta_{\theta_{opt}}$ the Dirac measure in $\theta_{opt}$, while considering the Hausdorff topology~\cite{HaarSaks91,HaarSaks92}.\\

\subsection{Definition, properties and convergence}
\noindent
In order to define the SA process and to present its main convergence result, the Dobrushin coefficient, its properties and some extra-notations are introduced.\\

\noindent
Let $\Theta=(\Theta,\cT,\upsilon)$ be a state space with $\upsilon$ a probability measure on $\cT$ and let $\Upsilon = \Upsilon(\Theta)$ be the set of all probability measures on the space $(\Theta,\cT)$. Throughout the entire, the norm $\parallel \cdot \parallel$ is the total variation norm, that is for any $\mu \in \Upsilon$:
\begin{eqnarray}
\label{totalvariation}
\parallel \mu \parallel & = & \sup_{A \in \cT} |\mu(A)| \nonumber\\
& = & \sup_{|g|< 1} |\mu(g)| = \sup_{|g|< 1} |\int_{\Theta} g(\theta)\mu(d\theta)|.
\label{eq:totalVariationNorm}
\end{eqnarray}

\noindent
Let $P$ be a transition kernel on $\Theta$. The Dobrushin coefficient is defined as :
\begin{equation}
\label{dobrushin-def}
c(P) = \ds\sup_{\theta,\psi \in \Theta} \parallel P(\theta,\cdot)- P(\psi,\cdot) \parallel=\ds\sup_{\theta,\psi\in \Theta}\sup_{A\in \cT}|P(\theta,A)-P(\psi,A)|.
\end{equation}

\begin{lemma} The Dobrushin coefficient, defined by \eqref{dobrushin-def}, verifies the following properties:
\begin{enumerate}[label=(\roman*)]
\item $0 \leq c(P) \leq 1$,
\item $\parallel \mu P - \lambda P \parallel \leq c(P)\parallel \mu - \lambda \parallel $,
\item $c(P_1 P_2) \leq c(P_1) c(P_2)$.
\end{enumerate}
\end{lemma}

\noindent
A proof of the previous result is given below. The reader may also refer to~\cite{Iosifescu69, VanLieshout94, makur16,Winkler03} for more details.\\

\begin{proof}
\noindent
$(i)$. Clearly $c(P)\geq 0$ by definition~\eqref{dobrushin-def}. The $c(P)\leq 1$ is also immediate by using the property that $P$ is a transition kernel. For all $A\in \cT$ and all $\theta, \psi \in \Theta $, $P(\theta, A), \, P(\psi, A)\in [0,1] $ so $|P(\theta,A) -P(\psi,A)|\leq 1.$ Thus, by taking the supremum in $A$ and $\theta, \psi$ it results that $c(P)\leq 1$.\\

\noindent
$(ii)$. Let us recall first a classical result. If $\mu$ and $\nu$ are from $\Upsilon$ and $g$ is a smooth  function we have, by using ~\eqref{kernelP}:
\begin{equation}
\label{inequality1}
|\mu(g)-\nu(g)| \leq \ds\frac{1}{2}\sup _{\theta,\psi\in\Theta }
|g(\theta) -g(\psi)| \, \|\mu-\nu\|.
\end{equation}
Let us prove now the inequality. By the definition of the total variation norm \eqref{eq:totalVariationNorm}, the formula~\eqref{kernelP} and the definition \eqref{dobrushin-def} we get:
\begin{equation}
\label{ineq1}
\begin{array}{ll}
\|\mu P - \nu P\|& = \ds\sup_{g, |g|\leq 1} |(\mu P) g - (\nu P)g| \\
& \leq \ds\sup_{g,|g|\leq 1}\left\{\ds\frac{1}{2}\sup_{\theta,\psi\in \Theta} |Pg(\theta) - Pg(\psi)| \|\mu -\nu\| \right\}\\
&\leq c(P) \|\mu-\nu\|,
\end{array}
\end{equation}
by applying \eqref{inequality1}.\\

\noindent 
$(iii)$. It can be directly obtained, by using \eqref{dobrushin-def} and $(ii)$ and lastly \eqref{inequality1}
\begin{equation*}
     \begin{array}{ll}
     c(P_1P_2) & = \ds\sup_{\theta, \psi \in \Theta} \|P_1P_2 (\theta,.)-P_1P_2(\psi,.)\|\\
     &=\ds\sup_{\theta, \psi \in \Theta} \|P_1 (\theta,.)P_2-P_1(\psi,.)P_2\|\\
     &\leq c(P_1)c(P_2).
     \end{array} 
 \end{equation*}
\end{proof}

\noindent
An immediate consequence of the previous result is that
\begin{equation}
\label{dobrushin-iteration}
c(P_1 P_2 \cdots P_n) \leq \prod_{i=1}^{n}c(P_i).   
\end{equation}

\noindent
Assume that $(P_j)_{j=1}^{\infty}$ is a sequence of transition kernels on $(\Theta,\cT)$ and that $\mu_0$ is a given initial distribution. The sequence  $(P_j)_{j=1}^{\infty}$ and the distribution $\mu_0$ define a discrete time {\it non-homogeneous Markov process} with the state space $\Theta$. Here, we are interested in the successive distributions $\mu_j = \mu_{j-1}P_j$ of the random variables $\theta_j$. Throughout the paper for $m+1\leq k$, the following notation is used
\begin{equation*}
P^{(m,k)} = P_{m+1}P_{m+2} \cdots P_{k},
\end{equation*}
so that one has $\mu_k = \mu_m P^{(m,k)}$.\\

\noindent
In the following it is assumed that the function to optimise $h : \Theta \rightarrow \RR$ is $\cT-$measurable, positive and bounded. For the sake of simplicity it is assumed that the global minimum of $h$ is scaled to zero. The global maximum is obtained by considering the function $-h$. It is denoted by 
\begin{equation}
    \label{minim}
M_h = \{\theta \in \Theta | h(\theta)  = 0\}
\end{equation} 
the {\it minimum set} of $h$ and by $\triangle h = \sup_{\theta \in \Theta} h(\theta)$ the maximum variation of $h$.\\

\begin{defn}
Let $h$ be the function to be optimised in the state space $(\Theta,\cT,\upsilon)$, $(P_j)_{j \in \NN}$ a sequence of transition kernels, $T_1 \geq T_2 \geq \ldots \geq T_{i-1} \geq T_{i} \rightarrow_{i \rightarrow \infty} 0$. A {\it simulated annealing } process is the non-homogeneous Markov process $(P_i)_{i\in \mathbb{N}}$ on $(\Theta,\cT,\upsilon)$ defined by
\begin{equation}
\label{transition-kernel}
\begin{array}{ll}
P_j(\theta,A) & = \ds\int_{A}\alpha_{j}(\theta \rightarrow
\psi)q(\theta \rightarrow \psi)\1_{\{\psi \in A\}}d\psi\\
& \qquad +  \1_{\{\theta \in A\}}
\left[1 - \ds\int_{A}\alpha_{j}(\theta \rightarrow
\psi)q(\theta \rightarrow \psi) d\psi \right],
\end{array}
\end{equation}
with
\begin{equation*}
\alpha_{j}(\theta \rightarrow \psi) = \min\left\{ 1,\left [\frac{\exp(-h(\psi))}{\exp(-h(\theta))}\right ]^{1/T_j} \frac{q(\psi \rightarrow \theta)}{q(\theta \rightarrow \psi)}
\right \}.
\end{equation*}
\end{defn}

\noindent
Again there is a lot of freedom in choosing the proposal distribution.  Similarly to~\eqref{idealProposal}, we consider the distribution
\begin{eqnarray}
q(\theta \rightarrow \psi) = q_\delta (\theta \rightarrow \psi |
\xx) = \frac{[f(\xx |\psi)/\zeta(\psi)]^{1/T_j}}{I_{j}(\theta, \delta,
\xx)}\1_{b(\theta, \delta/2)}\{\psi\} 
\label{idealProposalSA}
\end{eqnarray}
for a fixed $\delta > 0$, a parameter value $\nu \in \Theta$ and a realisation $\xx$ of the model $p(\cdot|\nu)$. The quantity $I_{j}(\theta, \delta, \xx)$ is given by the integral 
\begin{equation*}
I_{j}(\theta, \delta, \xx) = \int_{b(\theta, \delta/2)} \left[f(\xx|\phi)/c(\phi)\right]^{1/T_j} \, d \phi.
\end{equation*}
The SA process build with proposals~\eqref{idealProposalSA} is called {\it Ideal Simulated Annealing} (ISA).\\

\noindent
The following result states the convergence of the SA process for the optimisation of our problem, that is whenever $\upsilon(M_{h}) = 0$~(see \cite{HaarSaks91,HaarSaks92}), where $M_h$ is the set defined in \eqref{minim}.\\

\begin{thm} 
\label{theoreme1}
Let $(P_j)_{j\in\mathbb{N}}$ be a SA process. Suppose that there exist sequences $(n_j)_{j\in \mathbb{N}}$ and $(r_j)_{j\in\mathbb{N}}$  and a number $0~<d~<~1$ such that $(n_j)_{j\in \mathbb{N}}$ is increasing and $\lim_{j\rightarrow\infty} (j-r_j)= \infty=\lim_{j\rightarrow \infty} r_j$. Assume also that the following conditions hold: 
\begin{enumerate}[label=(\roman*)]
\item  For each $ j\geq 1$
\begin{equation} 
\label{cond1} 
 c(P^{(n_j,n_{j+1})}) \leq 1 - (1-d) \exp\left[- \ds\sum_{l=n_j+1}^{n_{j+1}} \ds\frac{\Delta h}{T_l}\right];
\end{equation}
\item$ \ds\lim_{k\rightarrow \infty} \ds\sum_{j=r_k}^k \exp\left(-\ds\sum_{l=n_j+1}^{n_{j+1}}\ds\frac{\Delta h}{T_l}\right) = \infty$;
\item $\ds\lim_{k\rightarrow\infty} \ds\frac{\cL_h(1/T_{n_{r_k}})}{\cL_h{(1/T_{n_{k+2}})}} =1$, where
$
\cL_{h}(1/T) = \int_{\Theta} \exp(-h(\theta)/T)\upsilon(d\theta). 
$
\end{enumerate}
Then
\begin{equation}
\ds\lim_{j\rightarrow \infty } \| \mu_j-\pi_j\|= 0 .
\end{equation}
\end{thm}

\noindent
The proof of this theorem is an adaptation of the proof given in \cite{HaarSaks91}.\\

\begin{proof}
\noindent
Suppose that $\varepsilon >0$ is given. The condition $(ii)$ implies that one can find integers $k$ and $k'$ such that for all $k>k'$ we have:
\begin{equation}
\label{eqthm11}
    \ds\sum_{j=r_k}^k \exp \left[-\Delta h \ds\sum_{l=n_j+1}^{n_{j+1}} \ds\frac{1}{T_l}\right]>\ds\frac{-\log(\varepsilon /4) }{1-d}.
\end{equation}
By using now the standard inequality $1-x< \exp(-x)$ for $x>0$ and the relation \eqref{eqthm11} it gets:
\begin{equation}
\begin{array}{l}
\label{eqthm12}
       \ds\prod_{j=r_k}^{k} \left[ 1-(1-d)\exp\left( -\Delta h \ds\sum_{l=n_j+1}^{n_{j+1}}\ds\frac{1}{T_l} \right)\right]\\ \qquad <\exp\left[-(1
       -d)\ds\sum_{j=r_k}^k \exp \left(-\Delta h\ds\sum_{l=n_j+1}^{n_{j+1}} \ds\frac{1}{T_l} \right)\right] \\ 
       \qquad <\ds\frac{\varepsilon }{4}.\\
       \end{array}
\end{equation}
The condition~$(i)$ is used to obtain
\begin{equation*}
c (P^{(n_j,n_{j+1})})\leq 1-(1-d)\exp\left(-\Delta h \ds\sum_{l=n_j+1}^{n_{j+1}} \ds\frac{1}{T_l}\right).
\end{equation*}
By doing now the product of these relations for $j\in \{ r_k, \ldots, k\}$ it results
\begin{equation}
\label{eqthm13}
\prod_{j=r_k}^k c(P^{(n_j,n_{j+1})}) \leq \prod_{j=r_k}^k\left[1-(1-d)\exp\left(-\Delta h \ds\sum_{l={n_j+1}}^{n_{j+1}} \ds\frac{1}{T_l}\right)\right].
\end{equation}
And now, by embedding \eqref{eqthm12} within \eqref{eqthm13} we obtain, for $k\ge k'$:
\begin{equation}
\label{eqthm14}
c(P_{n_{r_k}+1}\, P_{n_{r_k}+1} \,  \ldots \, P_{n_{k+1}}) < \ds\frac{\varepsilon}{4}. 
\end{equation}

\noindent
The previous integers $k$ and $k'$ can also be chosen such that 
\begin{equation}
\label{eq-Laplace}
\log\left({\cal{L}}_h\left(\frac{1}{T_{n_{r_k}}}\right)/{\cal{L}}_h\left(\ds\frac{1}{T_{n_{k+2}}}\right)\right) < \ds\frac{\varepsilon}{2},
\end{equation}
by using hypothesis $(iii)$.\\
Following~(\cite{HaarSaks91}, Thm. 5.1), for $0<k\leq i\leq n$ we have 
\begin{equation}
\label{control-pi}
\ds\sum_{i=1}^n\|\pi_i-\pi_{i-1}\|\leq 2 \log \left(\ds\frac{{\cal{L}}_h(1/T_k)}{{\cal{L}}_h(1/T_n)}\right).
\end{equation}
Thus by writing~\eqref{control-pi} with the adequate indexes and putting together the result in~\eqref{eq-Laplace}, it comes out that:
\begin{equation}
\label{control-pi-2}
\ds\sum_{l=n_{r_k}}^{n_{k+1}-1} \| \pi_l-\pi_{l-1}\| <\varepsilon.
\end{equation}

\noindent
Consider now two positive integers $M$ and $N$ such that $M\leq N-1$ and recall the notation, for $m\leq k-1$
\begin{equation*}
P^{(m,k)}=P_{m+1}P_{m+2}\ldots P_k.
\end{equation*}

\noindent
Let us prove by induction that
\begin{equation}
    \label{induction}
    \pi_MP^{(M,N)}-\pi_N=\ds\sum_{l=M}^{N-1}(\pi_l-\pi_{l+1})P^{(l,N)}.
\end{equation}

\noindent
For the verification step of this relation consider $M$ fixed and take $N=M+1$. This gives the following equality:
\begin{equation*}
    \pi_MP^{(M,M+1)}-\pi_{M+1}=(\pi_M-\pi_{M+1})P^{(M,M+1)},
\end{equation*}
which is verified since $\pi_{M+1}$ is a stationary distribution and thus $\pi_{M+1}=\pi_{M+1}P_{M+1}$.

\noindent
Assuming now, that the relation \eqref{induction} is verified for $M$ and $N$ fixed with $M\leq N-1$, let us
prove that~\eqref{induction} is also true for $N+1$. This leads to:
\begin{equation}
    \label{induction-step}
    \pi_MP^{(M,N+1)}-\pi_{N+1}=\ds\sum_{l=M}^N(\pi_l-\pi_{l+1})P^{(l,N+1)}
\end{equation}
which is equivalent to
\begin{equation*}
    (\pi_MP^{(M,N)}-\pi_{N+1})P_{N+1}=\ds\sum_{l=M}^{N-1}(\pi_l-\pi_{l+1})P^{(l,N)}P_{N+1}+\pi_NP_{N+1}-\pi_{N+1}P_{N+1},
\end{equation*}
that becomes
\begin{equation*}
    (\pi_MP^{(M,N)}-\pi_{N})P_{N+1}=\ds\sum_{l=M}^{N-1}(\pi_l-\pi_{l+1})P^{(l,N)}P_{N+1}
\end{equation*}
and this is true since~\eqref{induction-step} is satisfied. We admit now that \eqref{induction} is valid.

\noindent
Let us now complete the proof of the Theorem \ref{theoreme1}.

\noindent
Let $m>n_{k'+1}$ and take $k''$ such that $n_{k''+1}<m\leq n_{k''+2}$. 
As $(\pi_m)_{m\geq 1}$ is a stationary distribution, we have :
\begin{equation}
    \label{first-ineq}
    \begin{array}{ll}
    \|\mu_m-\pi_m\|&=\|\mu_{n_{r_{k''}}}P^{(n_{r_{k''}},m)}-\pi_m\|\\
    &\leq \|(\mu_{n_{r_{k''}}}-\pi_{n_{r_{k''}}})P^{(n_{r_{k''}},m)}\|+\|\pi_{n_{r_{k''}}}P^{(n_{r_{k''}},m)}-\pi_m\|.\\
    \end{array}
\end{equation}
In this last inequality the first term on the right hand side can be controlled by using \eqref{ineq1} and    \eqref{eqthm14}
\begin{equation}
    \label{second-ineq}
    \begin{array}{ll}
    \|(\mu_{n_{r_{k''}}}-\pi_{n_{r_{k''}}})P^{(n_{r_{k''}},m)}\|&\leq 2c(P^{(n_{r_{k''}},m)})\\
    &\leq 2c(P^{(n_{r_{k''}}, n_{k''+1})})\\
    &<\ds\frac{\varepsilon}{2}.
    \end{array}
\end{equation}
The second term can be expressed in the following form by using \eqref{induction}:
\begin{equation}
    \label{third-ineq}
    \begin{array}{ll}
    \|\pi_{n_{r_{k''}}}P^{(n_{r_{k''}},m)}-\pi_m\|&\leq \ds\sum_{l=n_{r_{k''}}}^{m-1}\|(\pi_{l+1}-\pi_l)P^{(l,m)}\|\\
    &^\leq \ds\sum_{l=n_{r_{k''}}}^{n_{k''+2}-1}\|\pi_l-\pi_{l+1}\|\\
    &<\ds\frac{\varepsilon}{2},
    \end{array}
\end{equation}
by making use of the inequality \eqref{control-pi-2}.
Combining now \eqref{second-ineq} and \eqref{third-ineq} in \eqref{first-ineq} gives finally for $m>n_{k''+1}$
\begin{equation*}
    \|\mu_m-\pi_m\|<\varepsilon.
\end{equation*}
As $\varepsilon>0$ is arbitrary the result of the theorem follows. This concludes the proof.
\end{proof}

\noindent
As a consequence of Theorem $\ref{theoreme1}$, it is obtained:\\

\begin{cor}
\label{corolaire1} 
Suppose that the conditions of Theorem \ref{theoreme1}
are satisfied and suppose also that we have the weak convergence $\ds\lim_{j\rightarrow+\infty} \pi_j= \pi$ where $\pi\in \Upsilon(\Theta) $.
Then we have also the weak convergence 
\begin{equation*}
\ds\lim_{j \rightarrow+\infty}\mu_j =\pi.
\end{equation*}
\end{cor}

\noindent
The following important result is also stated.\\

\begin{cor}
\label{corGlobalOptimum}
Suppose that the hypothesis of the Theorem \ref{theoreme1} are satisfied 
for the simulated annealing process $(P_i)_{i\geq 1}$. Suppose also that $h$ achieve its maximum in $\theta_{opt}\in \Theta $. Then we have the weak convergence 
$$\ds\lim_{j\rightarrow+\infty}\mu_j=\delta_{\theta_{opt}}.$$
\end{cor} 


\begin{rem}
The first two conditions in Theorem~\ref{theoreme1} ensure the weak ergodicity of the SA process. The condition (i) is fulfilled for transition kernels built with rather general proposal densities 
$q(\theta,\psi)\upsilon(d\psi)$
with $q : \Theta \times \Theta \rightarrow [0,\infty [$ measurable~(\cite{HaarSaks91}, Lemma 4.1, Thm. 4.2). The second condition (ii) if fulfilled if a logarithmic cooling schedule is used for the temperature
\begin{equation*}
T_{j} = \frac{K}{\log(j+2)} \quad \text{with} \quad K > 0,
\end{equation*}
which is similar to the cooling schedules obtained for maximising Markov random fields or marked point process probability densities~\cite{GemaGema84,StoiEtAl05}. The third condition (iii) is a bound for the sum of distances between two equilibrium distributions that correspond to two different temperatures~(\cite{HaarSaks91}, Thm. 3.2). It is the key condition, that whenever $\upsilon(M_{h}) = 0$, it allows to reduce the weak convergence of the process to the weak convergence of the equilibrium distributions~(\cite{HaarSaks92}, pp.44). 
\end{rem}

\noindent
The previous results show that the ISA process given by~\eqref{idealProposalSA} converges weakly towards the global optimum of the $h$ function~\eqref{functionH}. Still, these results cannot be transformed directly into an optimisation algorithm to be used in practice, due to the need of computation of the normalising constants $\zeta(\theta)$. The next section shows how to overcome this drawback by building an alternative chain able to follow the path of the theoretical ISA process as close as desired.\\

\section{Shadow Simulated Annealing process}
\noindent
The authors in~\cite{StoiEtAl17} proposed an algorithm, {\it ABC Shadow} able to sample from posterior distributions~\eqref{posteriorGibbs}. This section presents this sampling method, builds a SA process based on it and derives convergence results. This new process is called {\it Shadow Simulated Annealing} (SSA) process. 

\subsection{ABC Shadow sampling algorithm}
\noindent
The idea of the {\it ABC Shadow} algorithm~\cite{StoiEtAl17} is to construct a {\it shadow Markov chain} able to follow the {\it ideal Markov chain} given by~\eqref{transition-kernel}. The steps of the algorithm are given below.\\

\begin{algo}
\label{abcShadow}
{\bf ABC Shadow~:} Fix $\delta$ and $m$. Assume the observed pattern is $\yy$ and the current state is $\theta_0$.

\begin{enumerate}
\item Generate $\xx$ according to $p(\xx|\theta_0)$.
\item For $k=1$ to $m$ do
\begin{itemize}  
\item Generate a new candidate $\psi$ following the density $U_\delta(\theta_{k-1} \to \psi)$ defined by
\begin{equation*}
U_\delta (\theta \to \psi) = \frac{1}{V_\delta} \1_{b(\theta, \delta/2)}\{\psi\},
\label{uniformProposal}
\end{equation*}
with  $V_\delta$ the volume of the ball $b(\theta, \delta/2)$.

\item The new state $\theta_{k} = \psi$ is accepted with probability $\alpha_{s}(\theta_{k-1} \rightarrow \psi)$ given by
\begin{eqnarray}
\lefteqn{\alpha_{s}(\theta_{k-1} \rightarrow \theta_{k}) = }\nonumber \\
& = & \min\left\{1,\frac{p(\theta_{k}|\yy)}{p(\theta_{k-1}|\yy)}\times\frac{f(\xx | \theta_{k-1})\zeta(\theta_{k})\1_{b(\theta_{k},\delta/2)}\{\theta_{k-1}\}}
{f (\xx | \theta_{k})\zeta(\theta_{k-1})\1_{b(\theta_{k-1},\delta/2)}\{\theta_{k}\}}
\right\} \nonumber \\
& = & \min\left\{1,\frac{f(\yy|\theta_{k})p(\theta_{k})}{f(\yy|\theta_{k-1})p(\theta_{k-1})}\times\frac{f(\xx | \theta_{k-1})}
{f (\xx | \theta_{k})}
\right\}
\label{acceptance_probability_shadow}
\end{eqnarray}
otherwise $\theta_{k} = \theta_{k-1}$.
\end{itemize}
\item Return $\theta_m$.
\item If another sample is needed, go to step $1$  with $\theta_0 = \theta_n$.
\end{enumerate}
\end{algo}

\noindent
The transition kernel of the Markov chain simulated by the previous algorithm is given by, for every $A \in  \cT$
\begin{eqnarray*}
P_s(\theta,A) & = &\int_{A}\alpha_{s}(\theta \rightarrow
\psi)U_{\delta}(\theta \rightarrow \psi)\1_{\{\psi \in A\}}d\psi\\
& +  &\1_{\{\theta \in A\}}
\left[1 - \int_{A}\alpha_{s}(\theta \rightarrow
\psi)U_{\delta}(\theta \rightarrow \psi) d\psi \right].
\label{kernelShadow}
\end{eqnarray*}

\noindent
The authors in~\cite{StoiEtAl17} show also that since
\begin{equation*}
|P_i(\theta,A) - P_s(\theta,A)| \leq  K_1 \delta
\end{equation*}
with $K_1$ a constant depending on $\xx,p$ and $\Theta$, 
there exists $\delta_0 = \delta_{0}(\varepsilon,m) > 0$ such that for every $\delta \leq \delta_{0}$, we have
\begin{equation*}
|P_{i}^{m}(\theta, A) - P_{s}^{m}(\theta,A)| < \varepsilon
\end{equation*}
uniformly in $\theta \in \Theta$ and $A \in \cT$. If $p(x|\theta) \in \cC^{1}(\Theta)$, then a description of $\delta_{0}(\varepsilon, n)$ can be provided.\\

\noindent
The previous results state that for any $\varepsilon$ and a given $\xx$ we may find a $\delta$ such that the shadow and ideal chain may evolve as close as desired during a pre-fixed value of $n$ steps. Under these assumptions, if $m \rightarrow \infty$ the Algorithm~\ref{abcShadow} does not follow closely the ideal chain started in $\theta_0$, anymore. The ergodicity properties of the ideal chain and the triangle inequalities allow to give a bound for the distance the distance after $n$ steps, between the shadow transition kernel and the equilibrium regime
\begin{equation*}
\| P_s^{(m)}(\theta,\cdot) - \pi(\cdot) \| \leq M(\xx,\delta)\rho^m + \varepsilon.
\end{equation*}
with $\rho \in (0,1)$.\\

\noindent
Iterating the algorithm more steps it is possible by re-freshing the auxiliary variable. This mechanism allows to re-start the algorithm for $m$ steps more, and by this, to obtain new samples of the approximate distribution of the posterior.\\

\subsection{An SA process based on the shadow chain}
\subsubsection{Process construction and properties}
\begin{thm}
\label{thmSAShadow}
Let $(P_{i,j})_{j \geq 1}$ be an ISA process associated with the ideal transition kernel $P_i$~\eqref{kernel} that samples from $\pi \propto \exp(-h)$ with $h$ given by~\eqref{functionH} using the proposals~\eqref{idealProposal}. The cooling schedule $k_{T}$ is chosen with respect to the Theorem~\ref{theoreme1}. According to the Algorithm~\ref{abcShadow}, to each $(P_{i,j})$ a shadow chain $P_{s,j})$~\eqref{kernelShadow} is attached. Then, there exists a sequence $\{\delta_{j} = \delta_{j} (\varepsilon,m), j \geq 1\}$ such that
\begin{equation*}
|P_{i,j}^{m}(\theta, A) - P_{s,j}^{m}(\theta,A)| < \varepsilon
\end{equation*}
for all $j \geq 1$, uniformly in $\theta \in \Theta$ and $A \in \cT$. If $p(x|\theta) \in \cC^{1}(\Theta)$, then a description of $\delta_{j}(\varepsilon, n)$ can be provided.
\end{thm}

\begin{proof}
The proof is obtained by considering for each step~$j$ of the algorithm, the Proposition~1 in~\cite{StoiEtAl17}.
\end{proof}

\noindent
The previous process given by the sequence the shadow chains $P_s$ is named the Shadow Stochastic Annealing (SSA) process. Controlling the $\delta$ parameter in the same time with the temperature allows to obtain the following convergence properties.\\

\begin{thm}
Let us consider the assumptions of Theorem~\ref{thmSAShadow} fulfilled and let $\varepsilon >0$ be a fixed value. Then, for the SSA process $(P_{s,j})_{j \geq 1}$, the following results hold~:
\begin{enumerate}[label=(\roman*)]
\item For each $j, T_j, \delta_j$ we have $\|P_{s,j}^{m}-\pi_j\| \leq \varepsilon $ and also for $j$ big enough 
$\| P_{s,j}^{m} -\delta_{\theta_{opt}}\| \leq \varepsilon $.
\item Consider now for $j\geq 1$, $\varepsilon_j=\frac{\varepsilon}{j}$. For each $j$ we can construct $T_j, \delta_j$ such that  $\|P_{s,j}^{m}-\pi_j\|\leq \varepsilon_j$ and $\ds\lim_{j\rightarrow +\infty} \| P_{s,j}^{m}-\delta_{\theta_{opt}} \|=0$ 
\end{enumerate}
\end{thm}

\begin{proof}
Theorem~\ref{thmSAShadow} proves that there exists a sequence $\{\delta_{j} = \delta_{j} (\varepsilon,m), j \geq 1\}$ such that
\begin{equation*}
|P_{i,j}^{m}(\theta, A) - P_{s,j}^{m}(\theta,A)| < \varepsilon
\end{equation*}
for all $j \geq 1$, uniformly in $\theta \in \Theta$ and $A \in \cT$.\\ 

\noindent
Combining this with the result of the Corollary~\ref{corGlobalOptimum} allows to conclude.\\

\noindent
We end the demonstration by the following observation : if $p(x|\theta) \in \cC^{1}(\Theta)$, then a description of $\delta_{j}(\varepsilon, m)$ can be provided.
\end{proof}

\begin{rem}
The first part of the preceding result can be interpreted as follows. Let $\{\theta_j, j \geq 1 \}$ be a realisation of the SSA process. Then
there exists a sequence $\{\delta_{j} = \delta_{j} (\varepsilon,m), j \geq 1\}$ corresponding to each $\theta_j$ such that for $j \geq j_{\max} (\varepsilon,m)$
\begin{equation*}
\{\theta_j, j \geq j_{\max}\} \subset b(\theta_{opt},\varepsilon)
\end{equation*}
where $\theta_{opt}$ is the global optimum of the function $h$.
\end{rem}
\begin{cor}
We have 
\begin{equation}
\delta_j= \ds\frac{K(x_j,T_j, n_j)}{j}
\end{equation}
and $K(x_j, T_j, n_j)$
\end{cor}
\begin{rem}
The second part of the preceding result states that to a decreasing sequence of balls around the problem solution, it is possible to associate a sequence of $\delta$ parameters in order to get as close as desired to the global optimum of $h$.
\end{rem}

\subsubsection{A new algorithm for global optimisation}
\noindent 
The previous results justify the construction of the following algorithm.

\begin{algo}
\label{abcSAShadow}
{\bf Shadow Simulated Annealing (SSA) algorithm~:} fix $\delta=\delta_0$, $T=T_0$, $n$ and $k_{\delta}, k_{T} : \RR^{+} \rightarrow \RR^{+}$ two positive functions. Assume the observed pattern is $\yy$ and the current state is $\theta_0$.

\begin{enumerate}
\item Generate $\xx$ according to $p(\xx|\theta_0)$.
\item For $k=1$ to $m$ do
\begin{itemize}
\item Generate a new candidate $\psi$ following the density $U_\delta(\theta_{k-1} \to \psi)$ defined by
\begin{equation*}
U_\delta (\theta \to \psi) = \frac{1}{V_\delta} \1_{b(\theta, \delta/2)}\{\psi\},
\label{uniformProposal-1}
\end{equation*}
with  $V_\delta$ the volume of the ball $b(\theta, \delta/2)$.
\item The new state $\theta_{k} = \psi$ is accepted with probability $\alpha_{s}(\theta_{k-1} \rightarrow \psi)$ given by
\begin{eqnarray}
\lefteqn{\alpha_{s}(\theta_{k-1} \rightarrow \theta_{k}) = }\nonumber \\
& = & \min\left\{1,
\left[
\frac{p(\theta_{k}|\yy)}{p(\theta_{k-1}|\yy)}
\times
\frac{f(\xx | \theta_{k-1})}{f(\xx | \theta_{k}}
\right]^{1/T}
\times
\frac{\1_{b(\theta_{k},\delta/2)}\{\theta_{k-1}\}}
{\1_{b(\theta_{k-1},\delta/2)}\{\theta_{k}\}}
\right\} \nonumber \\
& = & \min\left\{1,
\left[
\frac{f(\yy|\theta_{k})p(\theta_{k})}{f(\yy|\theta_{k-1})p(\theta_{k-1})}\times\frac{f(\xx | \theta_{k-1})}
{f (\xx | \theta_{k})}
\right]^{1/T}
\right\}
\label{acceptance_probability_shadow_sa}
\end{eqnarray}
otherwise $\theta_{k} = \theta_{k-1}$.
\end{itemize}
\item Return $\theta_m$.
\item Stop the algorithm or go to step $1$  with $\theta_0 = \theta_n$, $\delta_0 = k_{\delta}(\delta)$ and $T_{0} = k_{T}(T)$.
\end{enumerate}
\end{algo}

\noindent
It is easy to see that the SSA algorithm is identical to Algorithm~\ref{abcShadow} whenever $\delta$ and $T$ remain unchanged. The SSA algorithm does not  have access at the states issued from the ideal chain. When the algorithm is started, the distance between the ideal and the shadow chain depends on the initial conditions. Still, independently of these conditions, as the control parameters $\delta_j$ and $T$ evolve, this distance evolves also, by approaching zero.


\section{Applications}
\noindent
This section illustrates the application of the SSA algorithm. The next part of this section applies the present method for estimating the model parameters of three point processes: Strauss, area-interaction and Candy model. All these models are widely applied in domains such environmental sciences, image analysis and cosmology~\cite{Lies00,MollWaag04,LiesStoi03,StoiDescZeru04,StoiEtAl07,StoiEtAl15,TempEtAl14,TempEtAl18}. Here the parameter estimation is done on simulated data and it is double aimed. The first purpose is to test the method on complicated models, that does not exhibit a closed analytic form for their normalising constants. The second one is to give to the potential user, some hints regarding the tuning of the algorithm. The last part of this section is dedicated to real data application: model fitting for the galaxy distribution in our Universe.\\

\subsection{Simulated data: point patterns and segment networks}
\noindent
The SSA Shadow algorithm is applied here to the statistical analysis of patterns which are simulated from a Strauss model~\cite{KellRipl76,Stra75}. This model describes random patterns made of points exhibiting repulsion. Its probability density with respect to the standard unit rate Poisson point process is
\begin{align}
p(\yy|\theta) & \propto \beta^{n(\yy)}\gamma^{s_{r}(\yy)}  \nonumber \\
              & =  \exp \left[ n(\yy) \log \beta  + s_r(\yy)\log \gamma \right].
\label{straussModel}
\end{align}
Here $\yy$ is a point pattern in the finite window $W$, while $t(\yy) = (n(\yy),s_{r}(\yy))$ and $\theta = (\log \beta, \log \gamma)$ are the sufficient statistic and the model parameter vectors, respectively. The sufficient statistics components $n(\yy)$ and $s_{r}(\yy)$ represent respectively, the number of points in $W$ and the number of pairs of points at a distance closer than $r$.\\

\noindent
The Strauss model on the unit square $W=[0,1]^2$ and with density parameters $\beta=100$, $\gamma=0.5$ and $r=0.1$, was considered. This gives for the parameter vector of the exponential model $\theta=(4.60,-0.69)$. The CFTP algorithm~(see Chapter 11 in~\cite{MollWaag04}) was used to get $1000$ samples from the model and to compute the empirical means of the sufficient statistics $\bar{t}(\yy)=(\bar{n}(\yy),\bar{s_{r}}(\yy))=(45.30,17.99)$. The SA based on the ABC Shadow algorithm was run using $\bar{t}(\yy)$ as observed data, while considering the $r$ parameter known.\\

\noindent
The prior density $p(\theta)$ was the uniform distribution on the interval $[0,7] \times [-7,0]$. Each time, the auxiliary variable was sampled using $100$ steps of a MH dynamics~\cite{Lies00,MollWaag04}. The $\Delta$ and $m$ parameters were set to $(0.01,0.01)$ and $200$, respectively. The algorithm was run for $10^6$ iterations. The initial temperature was set to $T_0=10^4$. For the cooling schedule a slow polynomial scheme was chosen 
\begin{equation*}
T_n = k_{T} \cdot T_{n-1}
\end{equation*}
with $k_T = 0.9999$. A similar scheme was chosen for the $\Delta$ parameters, with $k_{\Delta} = 0.99999$. Samples were kept every $10^3$ steps. This gave a total of $1000$ samples.\\

\noindent
The choice of the uniform prior was motivated by the fact that the posterior distribution using an uniform prior equals the likelihood distribution restricted to the domain of availability of the uniform distribution. Hence, following~\cite{BaddEtAl16,Lies00,MollWaag04} and the argument in~\cite{StoiEtAl17}, since the ML estimate approaches almost surely the true model parameters whenever the number of samples increases, the expected results of the SSA algorithm should be close to the model parameters used for its simulation.\\

\noindent
Figure~\ref{resultsStrauss} shows the time series of the SSA algorithm outputs applied to estimate the parameters of the previous Strauss process. The final values obtained for $\log\beta$ and $\log\gamma$ were $4.63$ and $-0.71$, respectively. The Table~\ref{tableStrauss} presents the empirical quartiles computed from the outputs of the algorithm. All these values are close to the true model parameters.\\

\begin{figure}[!htbp]
\begin{center}
\begin{tabular}{cc}
\includegraphics[width=6cm,height=4.5cm]{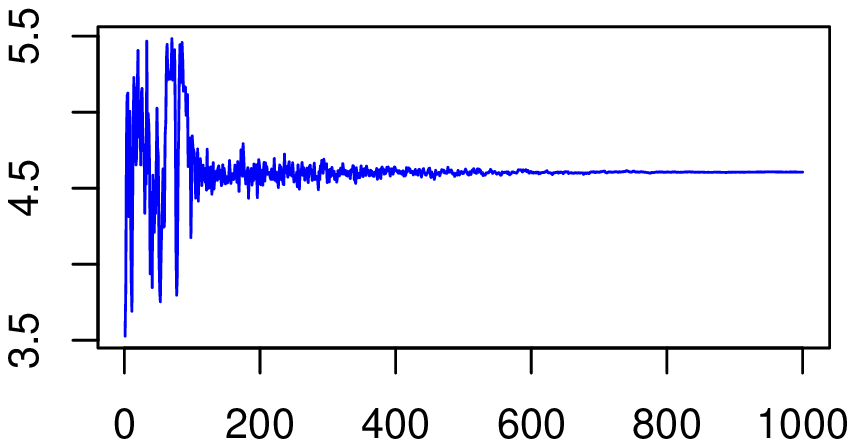} &
\includegraphics[width=6cm,height=4.5cm]{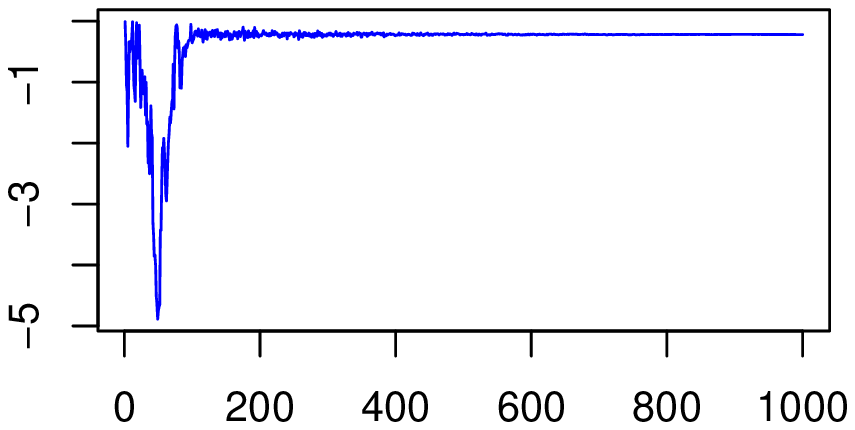}
\end{tabular}
\end{center}
\caption{SSA outputs for the MAP estimates computation of the Strauss model parameters. The true parameters of the model were $\theta=(4.60,-0.69)$, while the estimates are $\widehat{\theta}=(4.63,-0.71)$. }
\label{resultsStrauss}
\end{figure}

\begin{table}[!htbp]
\begin{center}
\begin{tabular}{|c|c|c|c|} 
\hline
\multicolumn{4}{|c|}{Summary statistics SSA Strauss estimation}\\
\hline
Parameters & $Q_{25}$ & $Q_{50}$ & $Q_{75}$ \\
\hline
SSA $\log\beta$  & 4.598 & 4.606 & 4.611\\
SSA $\log\gamma$ & -0.728 & -0.716 & -0.708\\
\hline
\end{tabular}
\end{center}
\caption{Empirical quartiles for the SSA Strauss model estimation.}
\label{tableStrauss}
\end{table}

\noindent
The area-interaction process introduced by~\cite{BaddLies95} is able to describe point patterns exhibiting clustering or repulsion. Its probability density with respect to the standard unit rate Poisson point process is
\begin{equation*}
p(\yy|\theta) \propto \beta^{n(\yy)}\gamma^{a_{r}(\yy)} = \exp \left[ n(\yy) \log \beta  + a_r(\yy)\log \gamma \right].
\label{areaInteractionModel}
\end{equation*}
with 
\begin{equation*}
a_{r}(\yy) = - \frac{\nu\left[\cup_{i=1}^{n}b(y_i,r)\right]}{\pi r^2}.
\end{equation*}

\noindent
The model vectors of the sufficient statistics and parameters are $t(\yy) = (n(\yy),a_{r}(\yy))$ and $\theta = (\log \beta, \log \gamma)$, respectively. If $\log \gamma < 0$ the point pattern surface induced by the radii around the points tends to occupy the whole domain $W$. This leads to a regular or repulsive distribution of points. If $\log \gamma > 0$ the point pattern surface induced by the radii around the points tends to be reduced. This leads to an aggregate or clustered distribution of points. Hence, parameter estimation of this model is also a morphological indicator, while sampling its posterior allows to assess statistical significance of this tendency~\cite{StoiEtAl17}.\\

\noindent
The following experiment was carried out, by considering the area-interaction model on the unit square $W=[0,1]^2$ and with density parameters $\beta=200$, $\gamma=e$ and $r=0.1$. This gives for the parameters vector of the exponential model $\theta=(5.29,1)$. A MH algorithm~(see Chapter 7 in~\cite{MollWaag04}) was used to get $1000$ samples from the model and to compute the empirical means of the sufficient statistics
$\bar{t}(\yy) = (\bar{n}(\yy),\bar{a_{r}}(\yy))=(144.31,-78.88)$. As previously, the SSA algorithm was run using $\bar{t}(\yy)$ as observed data, while considering the $r$ parameter known.\\

\noindent
The prior density $p(\theta)$ was the uniform distribution on the interval $[0,7] \times [-5,5]$. Each time, the auxiliary variable was sampled using $250$ steps of a MH dynamics~\cite{Lies00,MollWaag04}. The $\Delta$ and $m$ parameters were set to $(0.01,0.01)$ and $100$, respectively. The cooling schedule for the temperature, the descending scheme for $\delta$, the number of iterations and the number of samples were chosen as in the previous experiment.\\

\noindent
Figure~\ref{resultsAreaInt} and the Table~\ref{tableAreaInt} present the obtained results of the SSA algorithm applied to estimate the parameters of the previous area-interaction process. The final values of the algorithm's time series outputs were for $\log\beta$ and $\log\gamma$, $5.30$ and $1.03$, respectively. The empirical quartiles computed from the outputs of the algorithm indicate that the final results are close to the true model parameters.\\

\begin{figure}[!htbp]
\begin{center}
\begin{tabular}{cc}
\includegraphics[width=6cm,height=4.5cm]{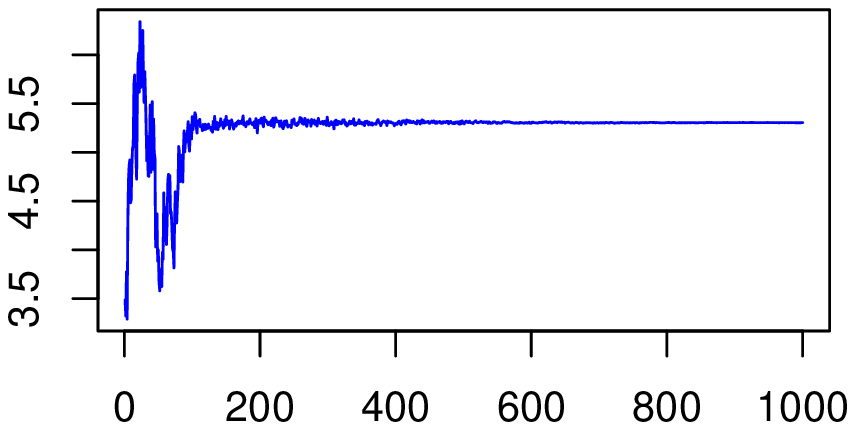} &
\includegraphics[width=6cm,height=4.5cm]{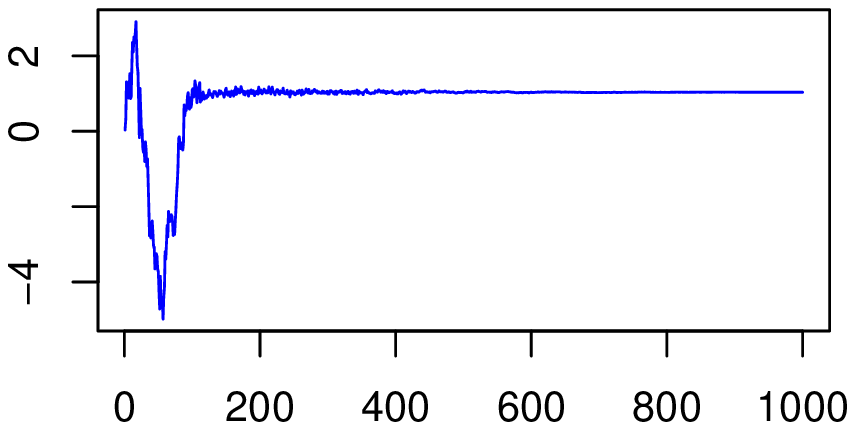}\\
\end{tabular}
\end{center}
\caption{SSA outputs for the MAP estimates computation of the area-interaction model parameters. The true parameters of the model were $\theta=(5.29,-1.00)$, while the estimates are $\widehat{\theta}=(5.30,-1.03)$.}
\label{resultsAreaInt}
\end{figure}

\begin{table}[!htbp]
\begin{center}
\begin{tabular}{|c|c|c|c|} 
\hline
\multicolumn{4}{|c|}{Summary statistics for SSA Area Interaction estimation}\\
\hline
Parameters & $Q_{25}$ & $Q_{50}$ & $Q_{75}$ \\
\hline
SSA $\log\beta$  & 5.298 & 5.304 & 5.308\\
SSA $\log\gamma$ & 1.026 & 1.035 & 1.042\\
\hline
\end{tabular}
\end{center}
\caption{Empirical quartiles for the SSA Area Interaction model estimation.}
\label{tableAreaInt}
\end{table}

\noindent
The Candy model is an object point process that simulates networks made of connected segments~\cite{LiesStoi03}. The model was successfully applied in image analysis and cosmology~\cite{StoiDescZeru04,StoiEtAl15}.

\noindent
A segment $y=(w,\xi,l)$ is given by its centre $w$, its orientation $\xi$ and its length $l$. The orientation is a uniform random variable on $M=[0,\pi)$, while the length is a fixed value. The probability density of the considered Candy model, with respect to the standard unit rate Poisson point process, is
\begin{equation}
p(\yy|\theta) \propto \exp \langle \theta_{d} n_{d}(\yy) + \theta_{s} n_s(\yy) + \theta_{f} n_f(\yy) + \theta_{r} n_{r}(\yy) \rangle
\label{candyModel}
\end{equation}
with $\yy$ a segments configuration, $\theta = (\theta_d,\theta_s,\theta_f,\theta_r)$ and 
\begin{equation*}
t(\yy) = (n_{d}(\yy), n_{s}(\yy), n_{f}(\yy), n_{r}(\yy))
\end{equation*}
the parameter and the sufficient statistic vectors, respectively. Each parameter controls its associate statistic. Here, $n_d$ is the number of segments connected at both of its extremities or doubly connected, $n_s$ is the number of segments connected at only one of its extremities or singly connected, $n_f$ is the number of segments that are not connected or free and $n_r$ the number of pairs of segments that are too close and not orthogonal.\\

\noindent
The connectivity of two segments is defined by the relative position of their extremities and their relative orientation. Two segments with only one  pair of extremities situated within the connection distance $r_c$ and with absolute orientation difference lower than a curvature parameter $\tau_c$ are connected. Similar to the orientation difference, the orthogonality of two segments is controlled the parameter $\tau_r$. For full details regarding the Candy model description and properties we recommend~\cite{LiesStoi03}.\\

\noindent
Figure~\ref{sampleCandy} pictures a realisation of the Candy model on $W=[0,3] \times [0,1]$. The segment length is $l=0.12$, the connection distance is $r_c = 0.01$, and the curvature parameters are $\tau_c=\tau_r=0.5$ radians. The model parameters are $\theta_d = 10$, $\theta_s=6$, $\theta_f=2$ and $\theta_r = -1$. It can be noticed that with these parameters the model outcomes tend to form patterns of a rather connected segments. An adapted MH 
algorithm~\cite{LiesStoi03}
was used to obtain $10000$ samples of the previous model and to compute the vector of the empirical means of the sufficient statistics $\bar{t}(\yy) = (\bar{n}_{d}(\yy) = 55.67, \bar{n}_{s}(\yy) = 50.26, \bar{n}_{f}(\yy) = 10.90, \bar{n}_{r}(\yy) = 40.24)$. These statistics were used as the data entry for a ABC SA algorithm.\\

\begin{figure}[!htbp]
\begin{center}
\includegraphics[width=8cm,height=5cm]{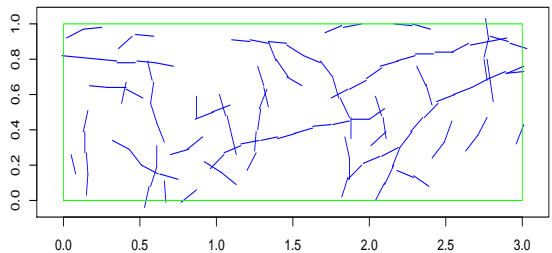} \end{center}
\caption{Realisation of the Candy model.}
\label{sampleCandy}
\end{figure}

\noindent
The prior density $p(\theta)$ was the uniform distribution on the interval $[0,12]^3 \times [-12,0]$. Each time, the auxiliary variable was sampled using $200$ steps of the adapted MH dynamics~\cite{LiesStoi06}. The $\Delta$ and $m$ parameters were set to $(0.01,0.01,0.01,0.01)$ and $500$, respectively. The other algorithm's parameters were chosen as in the previous examples.\\

\noindent
The algorithm results of the SSA are shown in 
Figure~\ref{resultsCandy} and  Table~\ref{tableCandy}. The final values of the algorithm's time series outputs were for $\theta_d$, $\theta_s$, $\theta_f$ and $\theta_r$, $10.009$,$6.002$,$1.982$ and $-0.996$ respectively. Together with the empirical quartiles given by the outputs of the algorithm, all these indicate that the algorithm outputs and the true model parameters are again rather close.\\

\begin{figure}[!htbp]
\begin{center}
\begin{tabular}{cc}
\includegraphics[width=6cm,height=4.5cm]{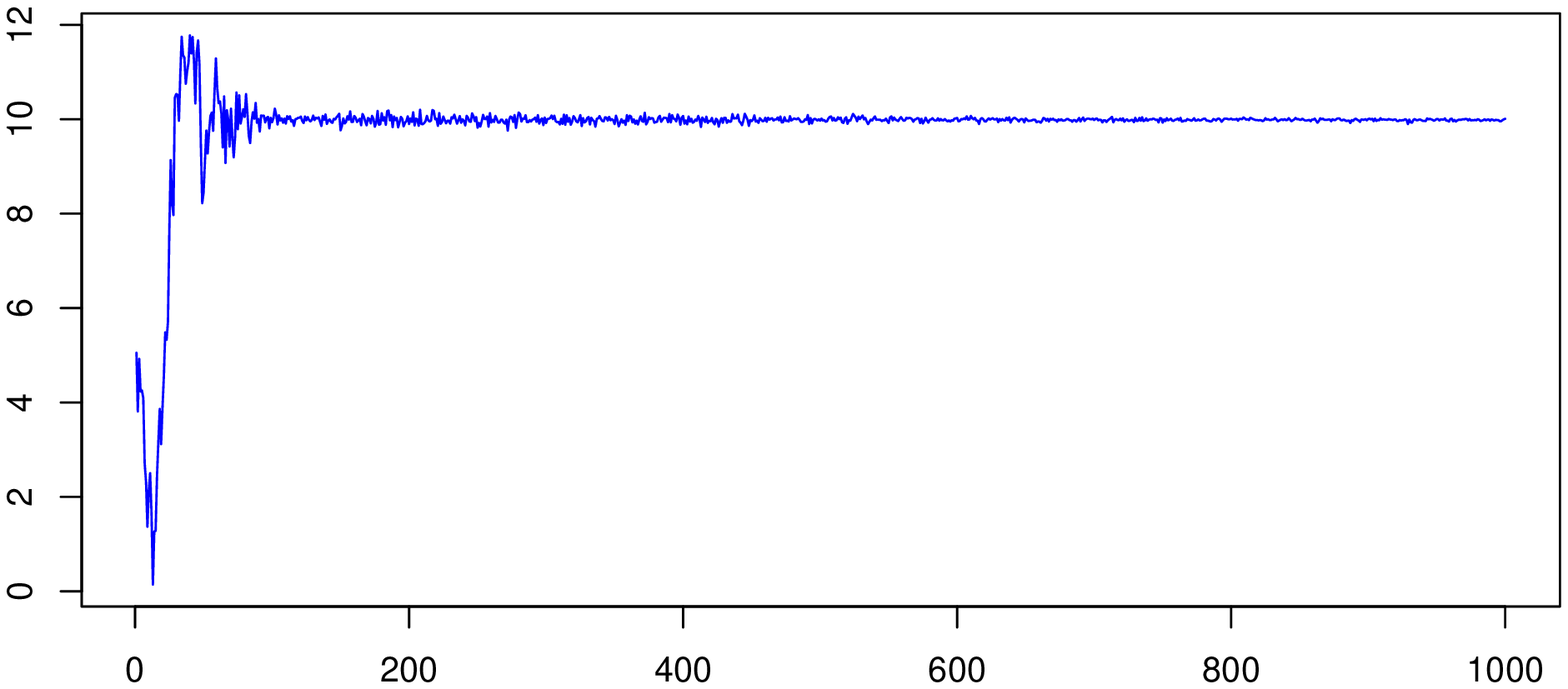} &
\includegraphics[width=6cm,height=4.5cm]{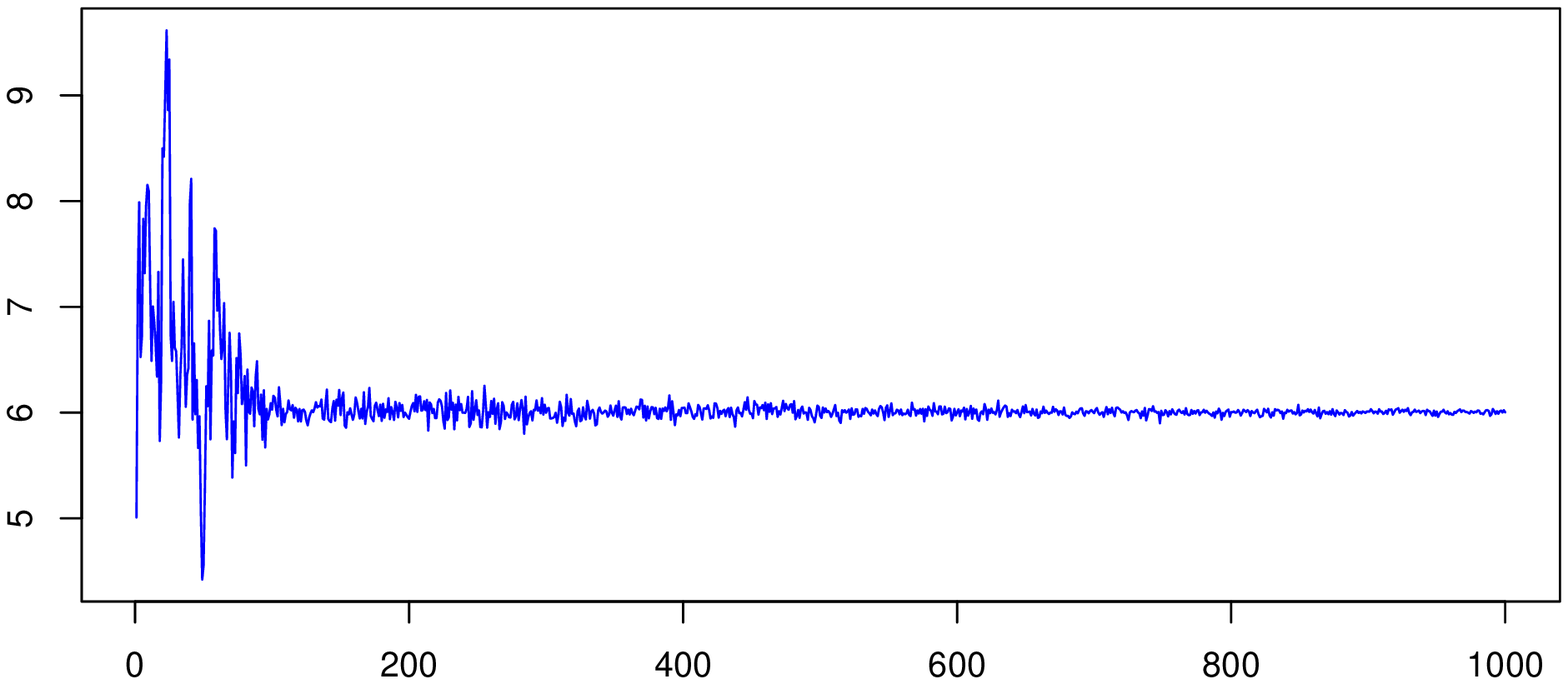}\\
\includegraphics[width=6cm,height=4.5cm]{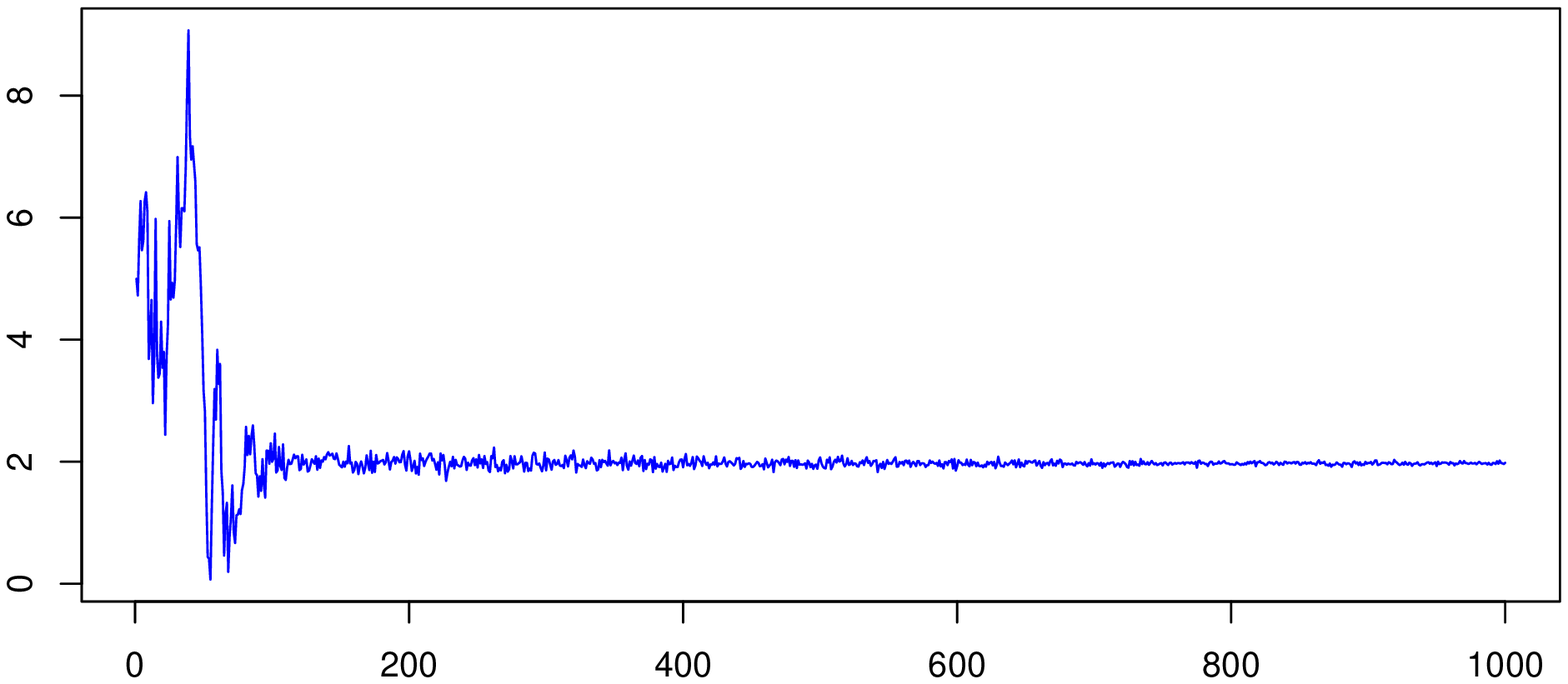} &
\includegraphics[width=6cm,height=4.5cm]{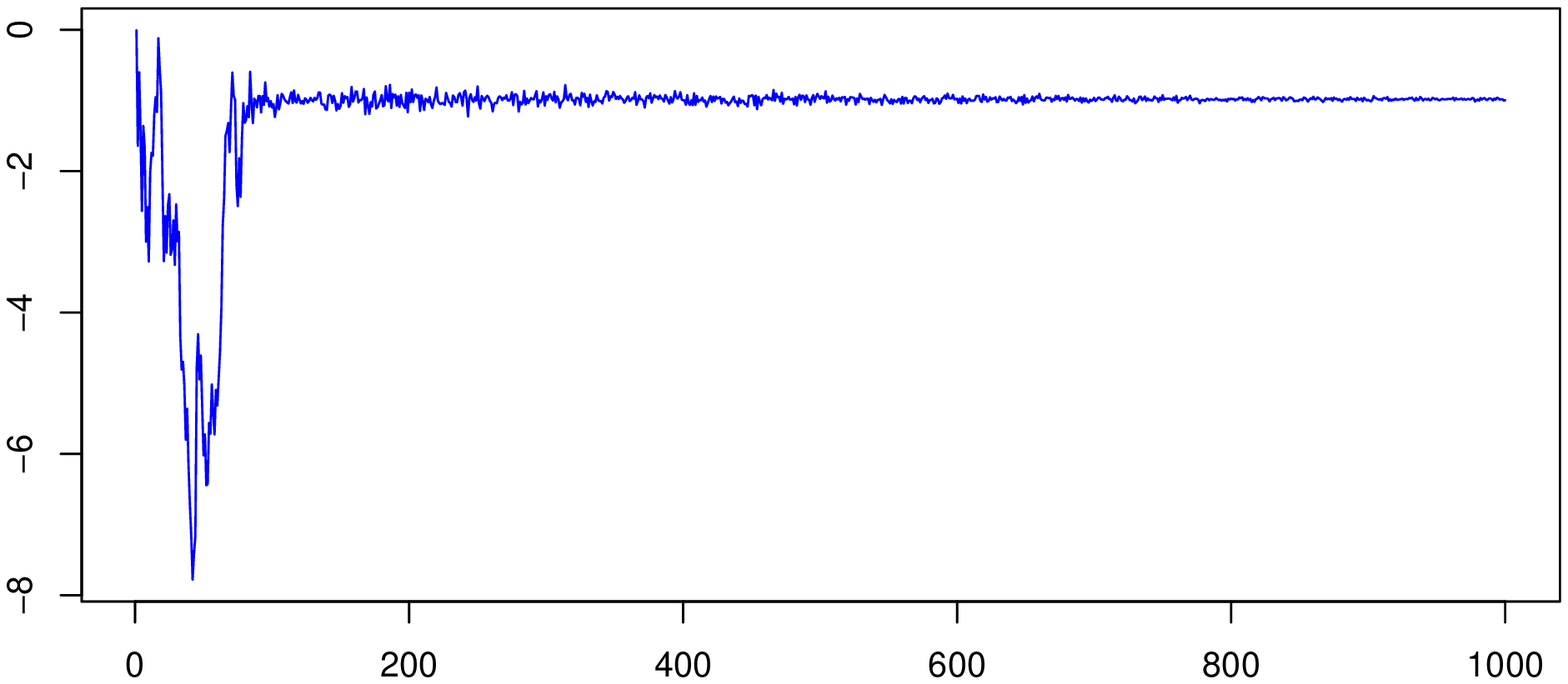}\\
\end{tabular}
\end{center}
\caption{SSA outputs for the MAP estimation of the Candy model parameters. The true parameters of the model were 
$\theta=(10,6,2,-1$), while the estimates are $\widehat{\theta}=(10.009,6.002,1.982,-0.996$).}
\label{resultsCandy}
\end{figure}

\begin{table}[!htbp]
\begin{center}
\begin{tabular}{|c|c|c|c|} 
\hline
\multicolumn{4}{|c|}{Summary statistics for the SSA Candy estimation}\\
\hline
Algorithm & $Q_{25}$ & $Q_{50}$ & $Q_{75}$ \\
\hline
SSA $\theta_d$ & 9.958 & 9.988 & 10.016\\
SSA $\theta_s$ & 5.983 & 6.009 & 6.040\\
SSA $\theta_f$ & 1.944 & 1.974 & 2.012\\
SSA $\theta_r$ & -1.017 & -0.985 & -0.962\\
\hline
\end{tabular}
\end{center}
\caption{Empirical quartiles for the SSA Candy model estimation.}
\label{tableCandy}
\end{table}

\subsection{Cosmology real data application: a point process model for fitting the galaxies distribution}
\noindent
The galaxies are not spread uniformly in our Universe. Their position exhibits an intricate pattern made of filaments and clusters~\cite{MartSaar02}. Knowing the positions of the galaxies, while assuming a Bayesian working framework, pattern detectors were built for filaments~\cite{StoiEtAl15,TempEtAl14,TempEtAl16} and more recently for clusters~\cite{TempEtAl18}.\\

\noindent
The previous detectors do not assume any particular model for the spatial distribution of galaxies. Now, given the detected structure, fitting a statistic model to the observed field of galaxies becomes a natural question. As a continuation of the application presented in~\cite{StoiEtAl17}, we show that the ABC Shadow posterior sampling algorithm and the SSA algorithm for parameter estimation are tools to be considered in such a task.\\

\noindent
Figure~\ref{galacticFilaments} pictures a sample of the considered data. It represents a cube of side $30$ $h^{-1}$ Mpc from SDSS (Sloan Digital Sky Survey) catalogue together with the induced filaments pattern~\cite{TempEtAl14}. 
Here the main axes of the filaments pattern are represented by a set of continuous curves that are called spines. More formally, the spines are ridge lines given by those regions where the filaments structures exhibited by the galaxies positions are the most aligned and the most connected. For full details regarding filaments detection and spines computation, the reader may refer to~\cite{TempEtAl14,TempEtAl16}.\\

\begin{figure}[!htbp]
\begin{center}
\includegraphics[width=14cm]{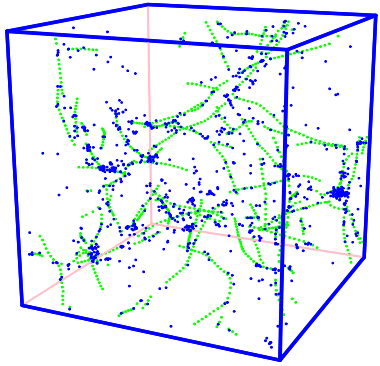} \end{center}
\caption{Galaxies positions (blue) and the induced filaments pattern or spines (green).}
\label{galacticFilaments}
\end{figure}

\noindent
Fitting a model to the galaxy distribution is extremely complex task~\cite{MartSaar02}. The authors~\cite{TempEtAl14,TempEtAl14a,TempEtAl16} suppose and infer that the galaxies tend to be close to the main axes of the filaments pattern, while forming clusters, similar to pearl on a necklace. Therefore, in the following, we consider a point process model that controls the number of galaxies, their proximity to the filaments network spine and their mutual interactions.\\

\noindent
Such a model can be represented by the point process given by the following probability density
\begin{equation}
p(\yy|\theta,F) \propto \beta_1^{n(\yy)} \beta_2^{d_{F}(\yy)} \gamma^{-a_{r}(\yy)}
\label{galaxyAreaInteraction}
\end{equation}
with the model parameters vector given by 
\begin{equation*}
\theta = (\log\beta_1,\log\beta_2,\log\gamma),    
\end{equation*}
and the sufficient statistics vector
\begin{equation*}
t(\yy) = (n(\yy),d_{F}(\yy), a_{r}(\yy)).
\end{equation*}

\noindent
The parameter $\beta_1 > 0$ controls the statistics $n(\yy)$ which represents the total number of galaxies in the configuration $\yy$. The parameter $\beta_2 > 0$ controls the proximity of the galaxies centres from the observed filaments pattern $F$. Its associate sufficient statistic is~:
\begin{equation*}
d_{F}(\yy) = - \sum_{i=1}^{n(\xx)} d(y_i,F)
\end{equation*}
with $d(y_i,F)$ the minimum distance from the galaxy position $y_i$ to the spines pattern $F$. The parameter $\gamma > 0$ produces repulsive ($\gamma < 1$) or clustering ($\gamma >1$) interactions among galaxies. The corresponding sufficient statistics is 
\begin{equation*}
a_{r}(\yy) = \frac{3A(\xx)}{4 \pi r^3},  A(\yy) = \nu[\bigcup_{i=1}^{n(\xx)}b(y_i,r)].
\end{equation*}
Here $b(y,r)$ is the ball centred in $y$ with radius $r$. So, $A(\yy)$ represents the volume of the object resulting from the set union of the spheres of radius $r$ and that are centred in the points given by the configuration $\yy$. The division of $A(\yy)$ by the volume of a sphere, allows to interpret the statistic $a_{r}(\xx)$ as the number of points or spheres needed to form a structure with volume $A(\yy)$.\\

\noindent
The proposed model~\eqref{galaxyAreaInteraction} is an inhomogeneous area-interaction process. The inhomogeneity governs the distribution of the galaxies with respect to the filaments pattern, while the area-interaction component controls the cluster formation among the galaxies. For further reading regarding the properties of the point process model we recommend~\cite{BaddLies95}.\\

\noindent
The sufficient statistics of the model are computed from the considered data set for various interaction radii. They are represented in the Table~\ref{tableGalaxy}. The total number of points in the observed volume and the distance to the filaments pattern do not depend on the interaction range.

\begin{table}[!htbp]
\begin{center}
\begin{tabular}{|c|c|c|c|c|c|c|c|} 
\hline
\multicolumn{8}{|c|}{Data for the galaxy pattern}\\
\hline
$r$            & 0.5 & 1 & 1.5 & 2 & 2.5 & 3 & 3.5\\
\hline
\multicolumn{8}{|c|}{$n(\yy) = 1024$, $d_{F}(\yy)= - \sum_{i=1}^{n(\yy)}d(y_i,F) = - 1180.05$}\\
\hline
$a_r(\yy)$ & 724.29 & 484.01 & 357.16 & 263.10 & 195.08 & 142.86 & 105.30\\
\hline
\end{tabular}
\end{center}
\caption{The observed sufficient statistics computed for the galaxy pattern, depending on the range parameter $r$.}
\label{tableGalaxy}
\end{table}

\noindent
First, for each radius value, the posterior distribution~\eqref{posteriorGibbs} associated to the model~\eqref{galaxyAreaInteraction} was maximised using the SSA Shadow algorithm. The prior density $p(\theta)$ was the uniform distribution on the interval $[-50,50]^3$. Each time, the auxiliary variable was sampled using $100$ steps of a MH dynamics~\cite{Lies00,MollWaag04}. The $\Delta$ and $m$ parameters were set to $(0.01,0.01,0.01)$ and $100$, respectively. The descending schedules for the temperature and the $\delta$ parameter, were fixed as in the previous examples. The algorithm was run for $10^6$ iterations. Samples were kept every $10^3$ steps. This gave a total of $1000$ samples.\\

\noindent
The obtained results are shown in the Figure~\ref{resultsDataSA}. The right column of the figure presents for each radius, the boxplots of the outputs of the SSA algorithm, associated to each model parameter. The model parameter behaviour can be analysed while $r$ increase. The $\beta_1$ parameter that is a baseline for the number of galaxies in the observed volume tend to stabilise. The evolution of the $\beta_2$ parameter indicates that the proximity to the filaments becomes more and more important while the considered interaction ranges increase. Almost the same thing can be stated for the behaviour of the $\gamma$ parameter. Still, there is a difference between the $\beta_2$ and the $\gamma$ parameters behaviour: the first one tends maybe to stabilise while the second one, exhibit maybe a decreasing tendency.\\

\noindent
The left column in the Figure~\ref{resultsDataSA} shows the evolution of the SSA algorithm for $r=2$. As for the previous cases, for high temperature values, the algorithm travels around the configuration space, while for low temperatures, it approaches the convergence regime. The obtained estimates for $\widehat{\theta}=(\widehat{\log\beta_1},\widehat{\log\beta_2},\widehat{\log\gamma)}$ are $(-0.33,0.98,4.57)$.\\

\begin{figure}[!htbp]
\begin{center}
\begin{tabular}{cc}
\includegraphics[width=6cm,height=4.5cm]{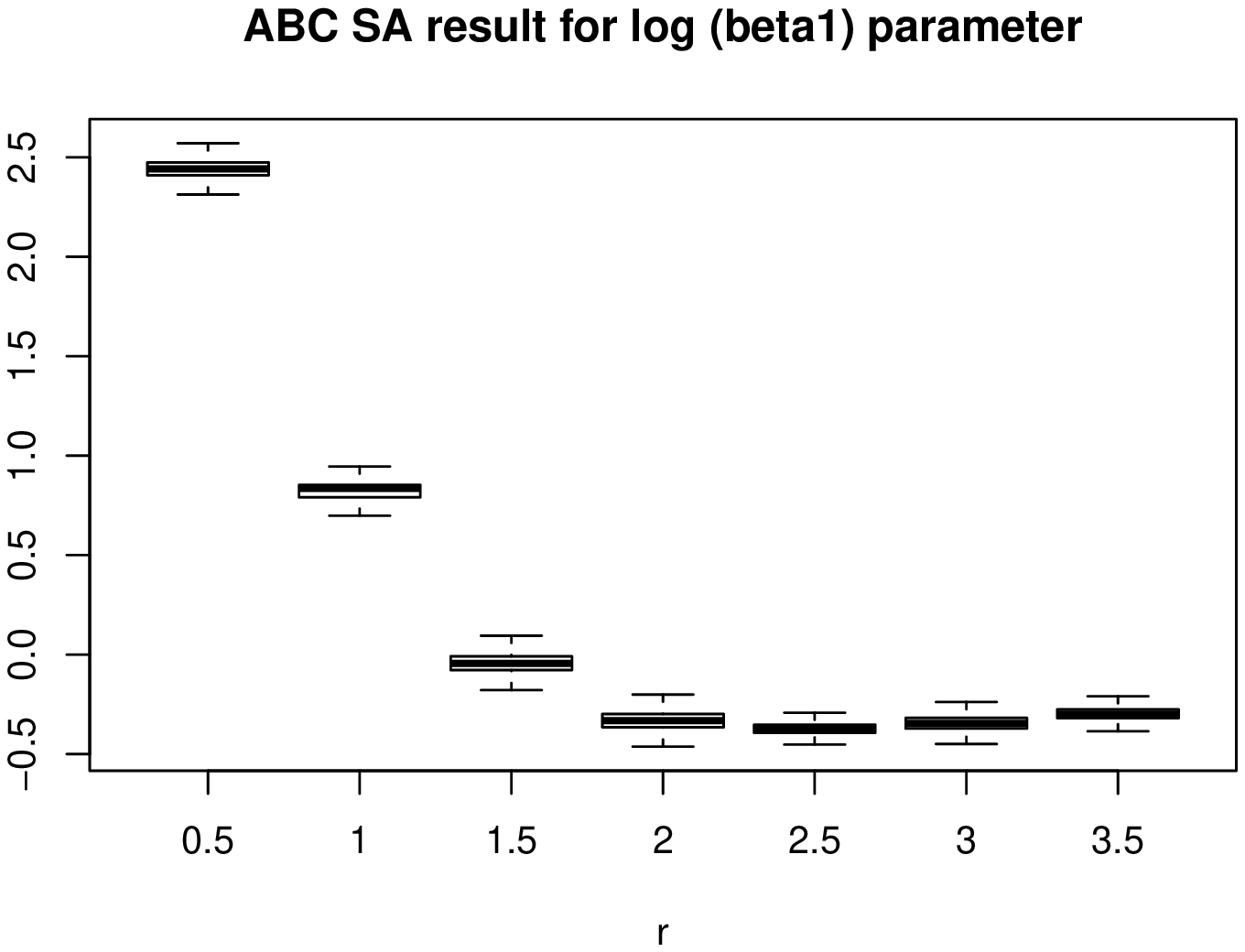} &
\includegraphics[width=6cm,height=4.5cm]{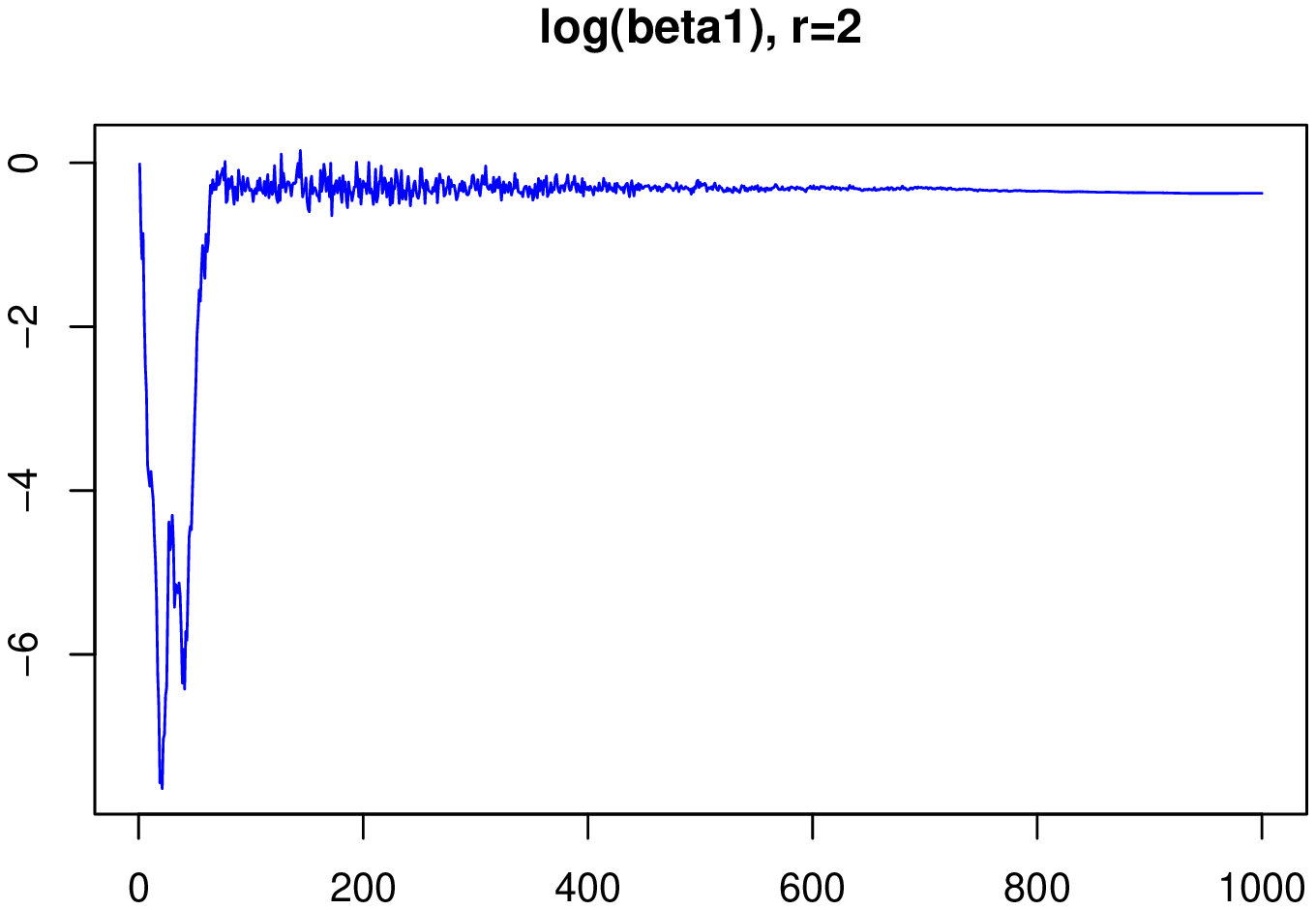}\\
\includegraphics[width=6cm,height=4.5cm]{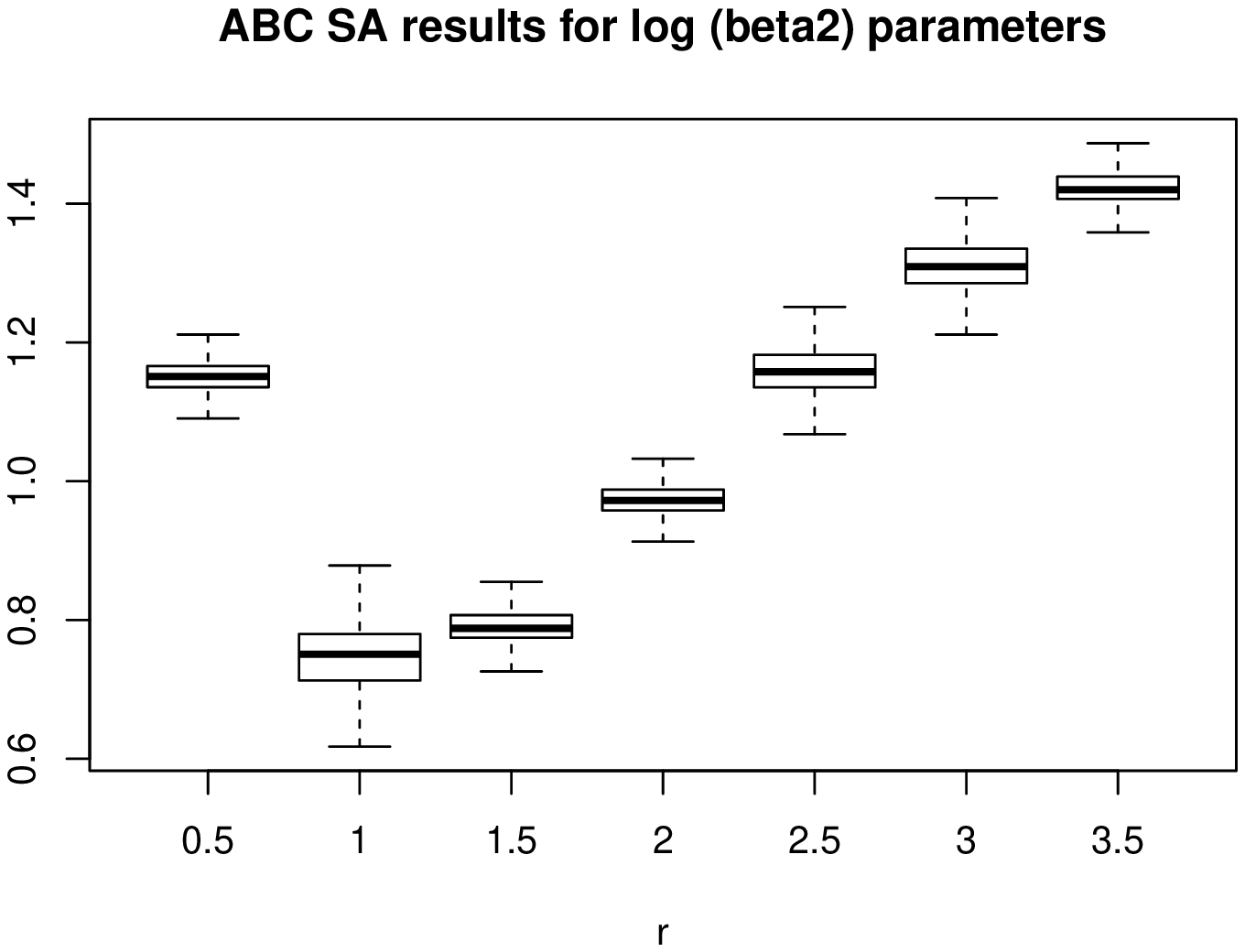} &
\includegraphics[width=6cm,height=4.5cm]{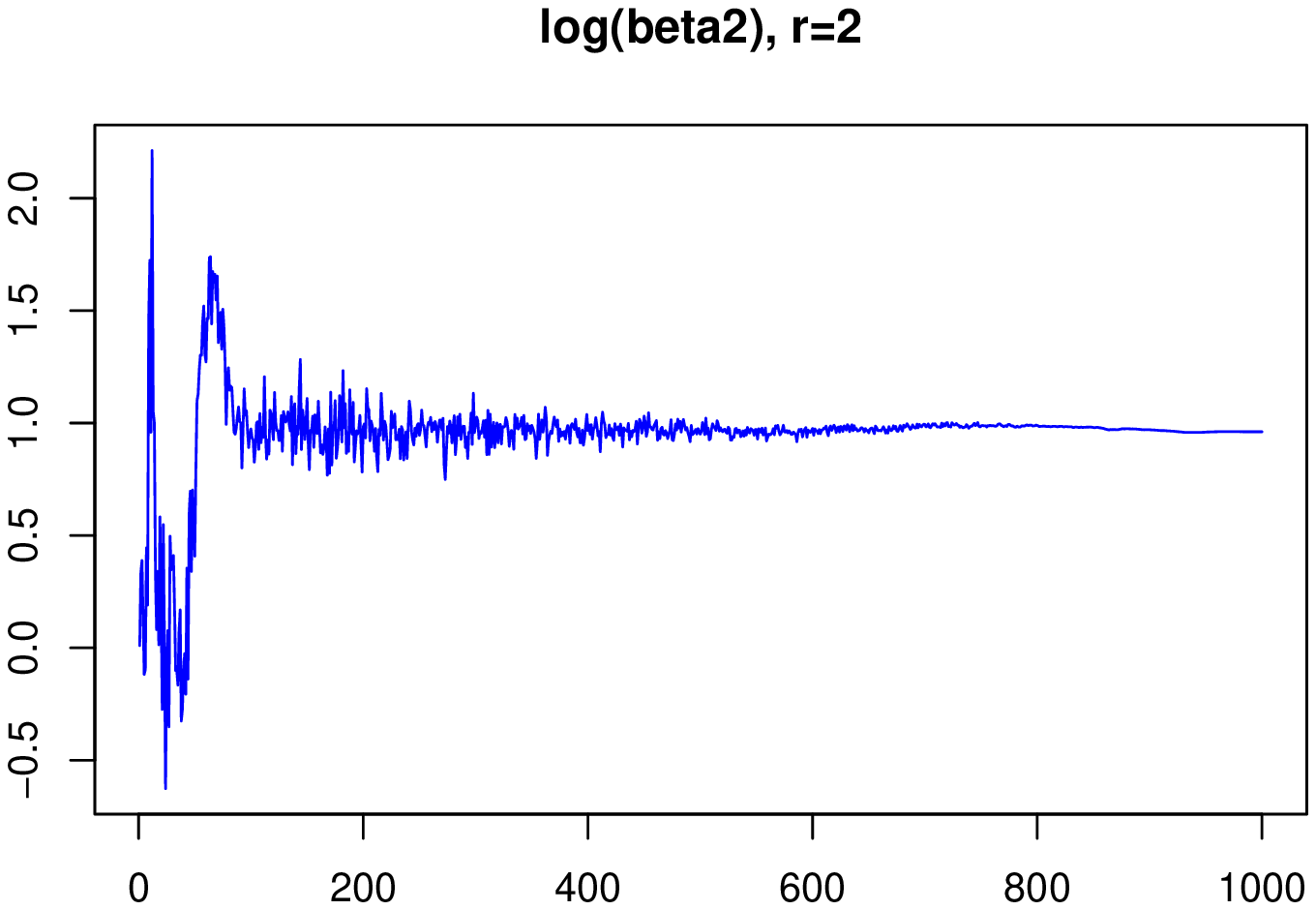}\\
\includegraphics[width=6cm,height=4.5cm]{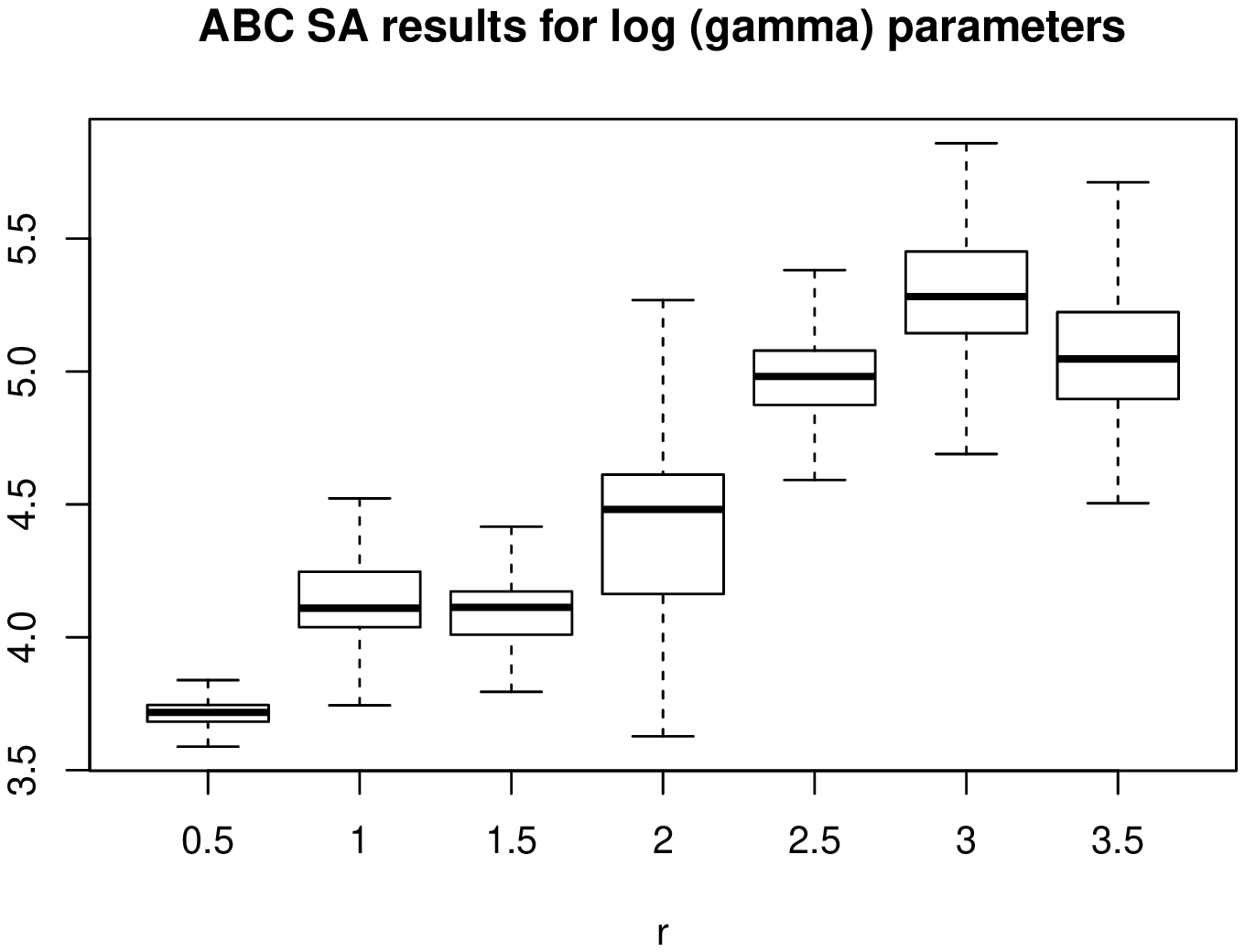} &
\includegraphics[width=6cm,height=4.5cm]{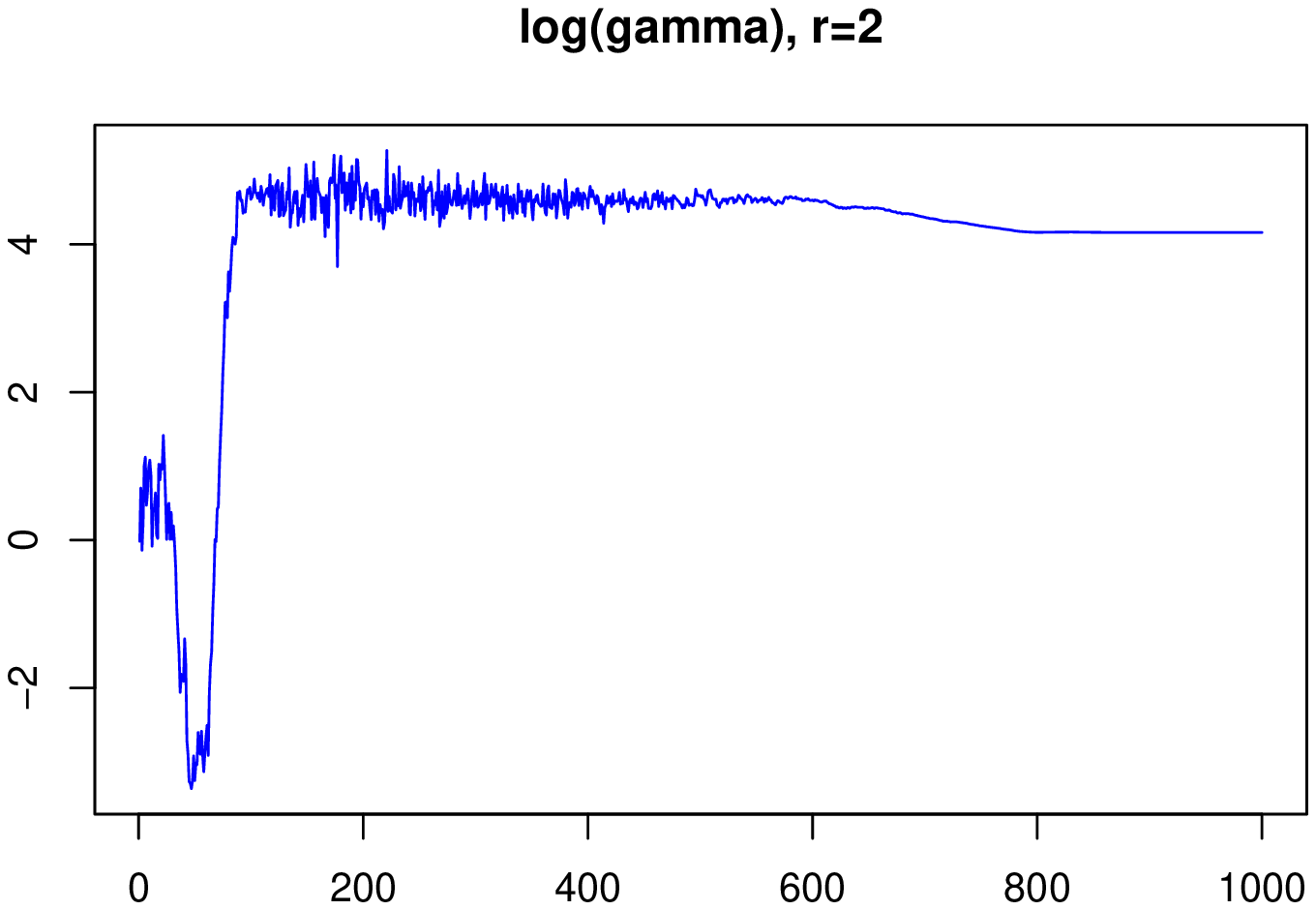}\\
\end{tabular}
\end{center}
\caption{SSA outputs for the MAP estimates computation of the inhomogeneous area-interaction model parameters fitted to the considered SDSS sample. Right column:  box plots of SSA algorithm outputs for each parameter depending on the interaction radius. Left column: the time series of the outputs of the SSA algorithm for $r=2$. }
\label{resultsDataSA}
\end{figure}

\noindent
Since the chosen prior distribution $p(\theta)$ was the uniform distribution over the compact parameter space $\Theta$, the SSA output is the Maximum Likelihood Estimate (MLE) restricted to $\Theta$. Hence, asymptotic errors may be derived following~\cite{Geye99,LiesStoi03,StoiEtAl17}. For each parameter, the asymptotic standard deviation and the Monte Carlo Standard error (MCSE) were computed respectively. The asymptotic standard deviation indicates the difference between the MLE estimate and the true model parameters. The SSA output can be also interpreted as a MCMCML estimate. Then the MCSE represents the difference between the MLE and its Monte Carlo counterpart. The obtained results are presented in Table~\ref{asymptoticsTable}, where for each parameter the first column indicates the asymptotic standard deviation and the second one, the MCSE. These values indicate a rather high quality of the estimation for models with interaction radius smaller than $3$ $h^{-1}$ Mpc. This may be explained by the fact that maybe a change of the clustering regime appears at these interaction ranges~\cite{EinaEtAl12}. Such a change is also suggested in Figure~\ref{resultsDataSA} where the interaction parameter behaves like reaching a maximum around $3$ $h^{-1}$ Mpc. 

\begin{table}[!htbp]
\begin{center}
\begin{tabular}{|c|c c|c c|c c|} 
\hline
\multicolumn{7}{|c|}{Asymptotic errors for the SSA estimates}\\
\hline
$r$ & $\sigma_{\log\beta_1}$ & $\sigma_{\log\beta_1}^{MC}$ & $\sigma_{\log\beta_2}$ &$\sigma_{\log\beta_2}^{MC}$ &
$\sigma_{\log\gamma}$ & $\sigma_{\log\gamma}^{MC}$ \\
\hline
0.5 & 0.04 & 1e-4 & 0.03 & 7e-5 & 0.07 & 2e-4\\
1 & 0.03 & 1e-4 & 0.03 & 9e-5 & 0.09 & 5e-4\\
1.5 & 0.05 & 1e-4 & 0.04 & 1e-4 & 0.12 & 1e-3\\
2 & 0.05 & 1e-4 & 0.04 & 2e-4 & 0.17 & 2e-3\\
2.5 & 0.05 & 2e-4 & 0.04 & 2e-4 & 0.24 & 5e-3\\
3 & 0.05 &  2e-4 & 0.04 & 5e-4 & 0.43 & 0.016\\
3.5 & 0.05 & 2e-4 & 0.04 & 9e-4 & 0.66 & 0.039\\
\hline
\end{tabular}
\end{center}
\caption{Asymptotics errors for the SSA MAP estimates of the model~\eqref{galaxyAreaInteraction} fitted to the considered cosmological sample. For each corresponding radius a model was fitted, and the asymptotic errors were computed for each model. For the computation of the MCSE $15\times10^3$ samples from the fitted model were used.}
\label{asymptoticsTable}
\end{table}

\noindent
In order to test the values and the significance of each of the model parameters, the ABC Shadow algorithm was used as in~\cite{StoiEtAl17}. The algorithm was run for $10^5$ iterations. Samples were kept every $100$ steps. This gave a total of $1000$ samples.\\

\noindent
The results obtained for the interaction radius $r=2$ are shown in the Figure~\ref{histDataABC}. From the obtained results, the marginals of each parameter posterior are approximated using an Epanechnikov kernel, and on this basis, the computed MAP is $\widehat{\theta}=(-0.29,0.95,4.57)$. The results are coherent since they are originated from an approximate distribution while fitting into the confidence intervals induced by the SSA estimation and its associate asymptotic standard error.\\

\begin{figure}[!htbp]
\begin{center}
\begin{tabular}{c}
\includegraphics[width=6cm,height=4.5cm]{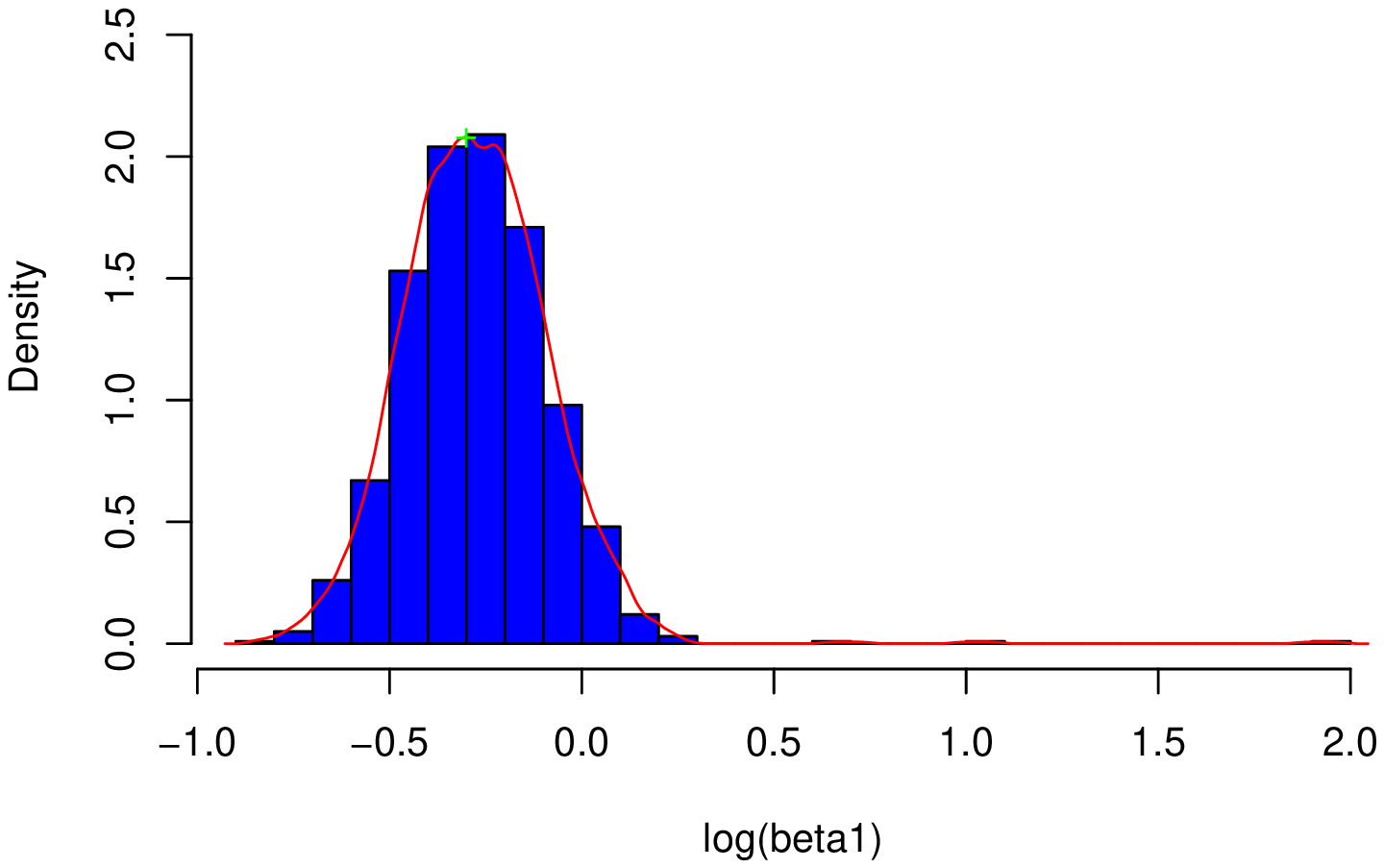}\\
\includegraphics[width=6cm,height=4.5cm]{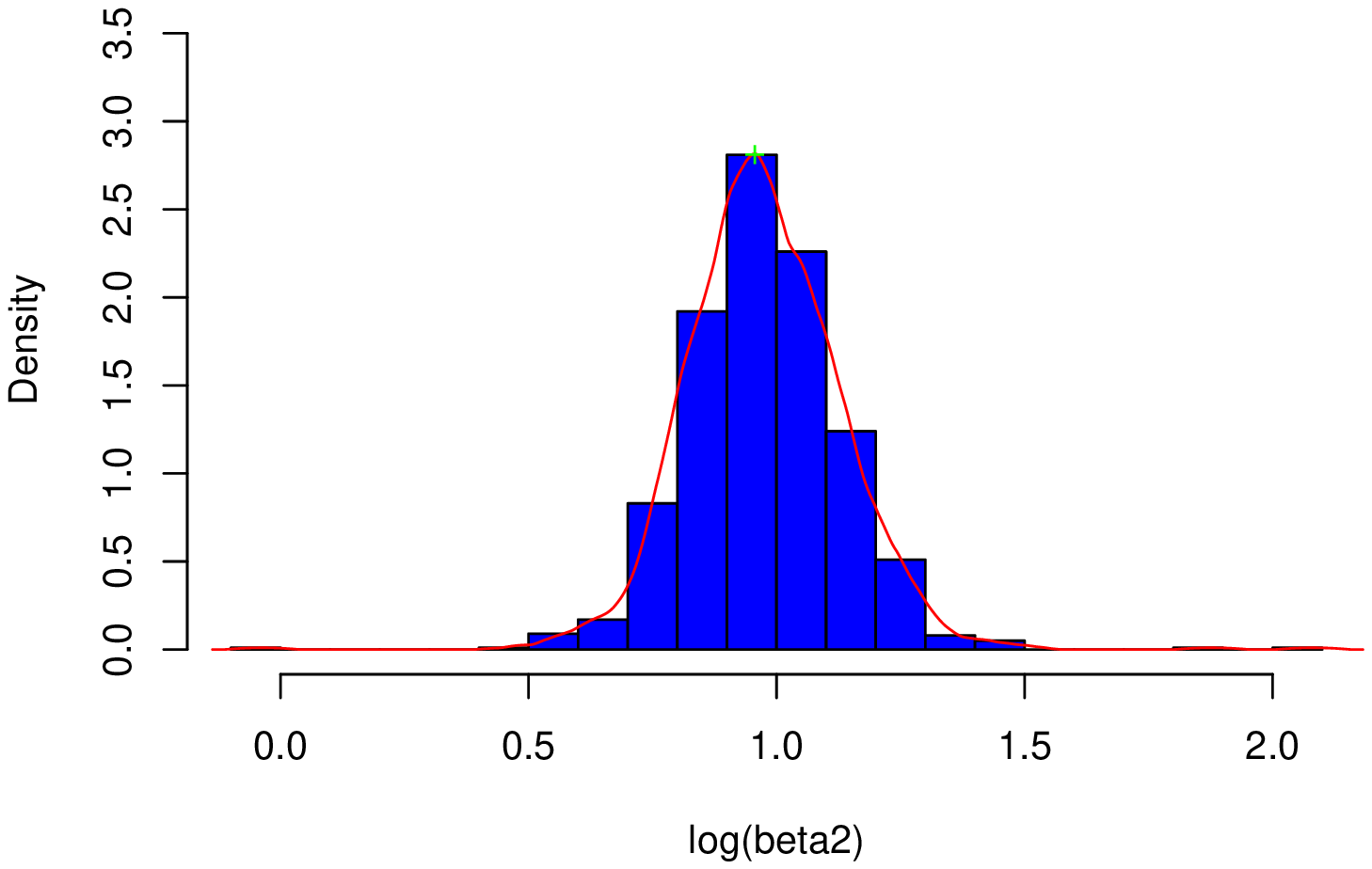}\\
\includegraphics[width=6cm,height=4.5cm]{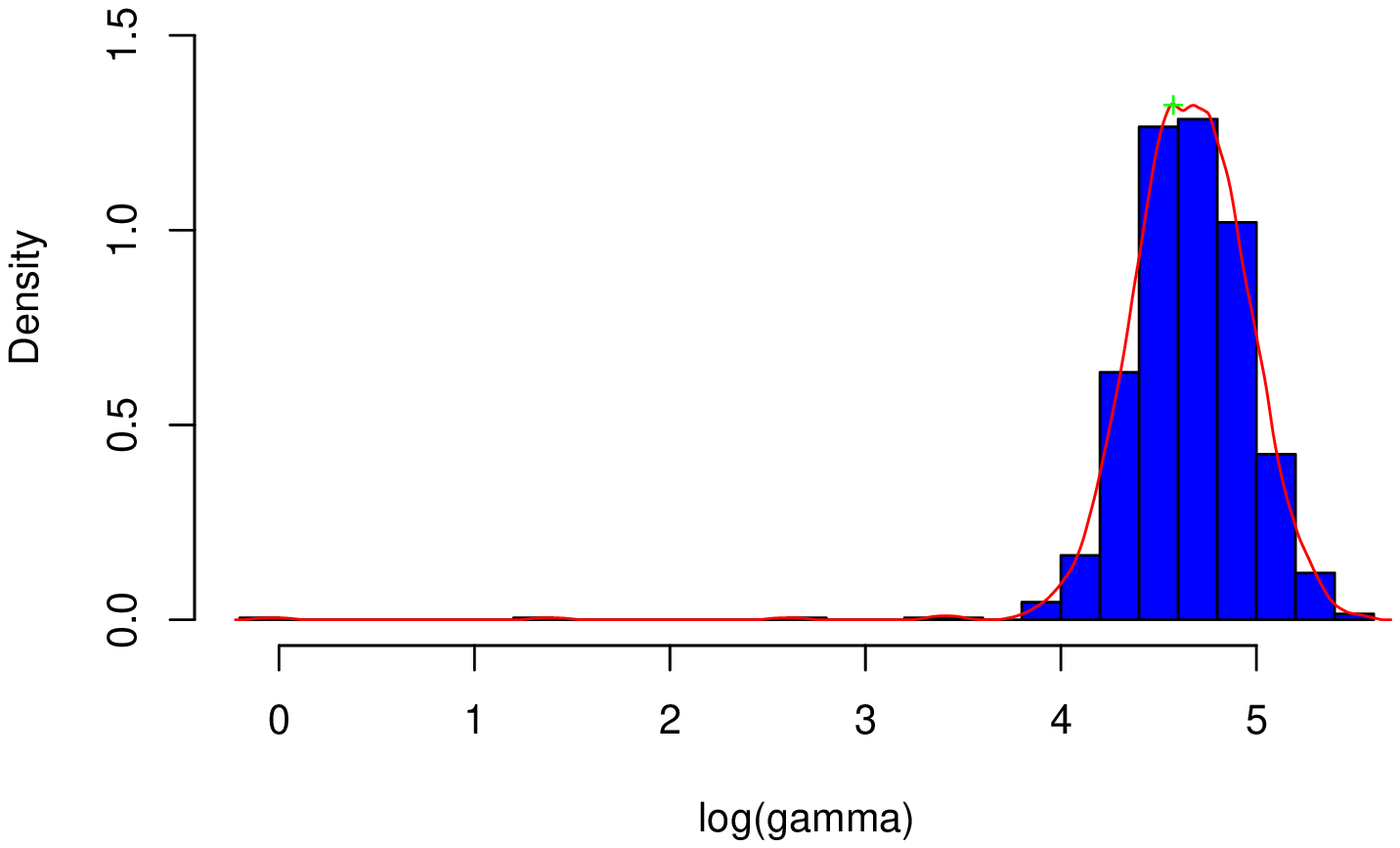}\\
\end{tabular}
\end{center}
\caption{ABC Shadow outputs for the approximate posterior sampling of the inhomogeneous area interaction model fitted to the SDSS sample (the range parameter is $r=2$).  The obtained MAP estimate is $\widehat{\theta}=(-0.29,0.95,4.57)$.}
\label{histDataABC}
\end{figure}

\noindent
In the following, based on the ABC Shadow approximation of the posterior of~\eqref{galaxyAreaInteraction} described previously, two statistical tests were conducted.\\

\noindent
First, a Student test was carried on to check whether the mean of the posterior distribution is different from the the SSA algorithm output. This test was conducted for each parameter, respectively. It used the marginal posterior samples given by the ABC Shadow. The obtained $p-values$ were all greater than $0.77$, so there is no evidence that the approximated posterior mean is different from the SSA output. This is a rather encouraging result, since the ABC Shadow is an approximate method, while the SSA exhibits convergence.\\

\noindent
Next, a second Student test was conducted in order to verify whether the obtained parameter values are significantly different from $0$. The obtained $p-$value for $\log\beta_1$ was $0.16$. For the parameters $\log\beta_2$ and $\log\gamma$ the corresponding $p-values$ were less than $10^{-5}$. The result indicates that the $\beta_1$ parameter is not significantly different from $1$. It also gives statistical significance of the following cosmological facts: the galaxies tend to be distributed close to the filaments while forming clusters, and only rarely being placed independently in our Universe.\\

\section{Conclusion and perspectives}
\noindent
This paper presents a global optimisation method, the Shadow Simulated Annealing (SSA) algorithm, that can be applied to a family of criteria that are not fully known. It can be applied to maximise posterior probability densities exhibiting normalising constants that are not available in analytic closed form. The SSA algorithm can be used with those posteriors provided by probability densities that are continuously differentiable with respect to their parameters. This is rather a strong hypothesis but often encountered in practice.\\

\noindent
The method provides convergence results towards the global optimum, that are equivalent with the results obtained for the simulated annealing whenever the optimisation criterium is entirely known~\cite{GemaGema84,HaarSaks91,HaarSaks92}.\\

\noindent
This work opens perspectives from a mathematical and application point of view. The mathematical challenge is to extend the family of criteria to which these methods apply. From an applied point of view, a thorough statistical study of the galaxies distribution in our Universe, based on the tools presented in this work, it is currently carried on by part of the authors of the paper.\\

\section*{Aknowledgements}
The authors are grateful to Elmo Tempel for providing the SDSS data set sample. Part of the work of the first author was supported by a grant of the Ministry of National Education and Scientific Research, RDI Program for Space Technology and Advanced Research - STAR, project number
513.


\bibliographystyle{plain}
\bibliography{sapp}

\end{document}